\documentclass[11p]{article}

\usepackage{amsfonts}
\usepackage{amsmath}
\usepackage{amsthm}
\usepackage{amssymb} 

\usepackage{mdframed} 
\usepackage{thmtools}
\usepackage{enumitem}

 \usepackage{xcolor}
\usepackage{color}
\usepackage{graphicx}

\usepackage{algorithm}
\usepackage[noend]{algpseudocode}
\usepackage{verbatim}

\usepackage[algo2e]{algorithm2e} 

\newcommand{\R}{\mathbb{R}} 

\newcommand{\cD}{{\cal D}}

\newcommand{\cI}{{\cal I}}

\newcommand{\cO}{{\cal O}}

\newcommand{\mA}{{\bf A}}
\newcommand{\mB}{{\bf B}}

\newcommand{\mD}{{\bf D}}
\newcommand{\mE}{{\bf E}}

\newcommand{\mI}{{\bf I}}

\newcommand{\mM}{{\bf M}}

\newcommand{\mP}{{\bf P}}
\newcommand{\mQ}{{\bf Q}}

\newcommand{\mS}{{\bf S}}

\usepackage[colorinlistoftodos,bordercolor=orange,backgroundcolor=orange!20,linecolor=orange,textsize=scriptsize]{todonotes}

\newcommand{\eqdef}{:=} 

\DeclareMathOperator{\argmin}{argmin}        

\DeclareMathOperator{\diag}{diag}       

\newcommand{\Diag}[1]{\mathbf{Diag}\left( #1\right)}

\providecommand{\trace}[1]{{\rm Trace}\left( #1\right)}

\newcommand{\Prob}{\mathbb{P}}

\newcommand{\E}[1]{\mathbb{E}\left[#1\right] }

\declaretheorem[style=shaded,within=section]{definition}
\declaretheorem[style=shaded,sibling=definition]{theorem}

\declaretheorem[style=shaded,sibling=definition]{assumption}

\declaretheorem[style=shaded,sibling=definition]{lemma}

\theoremstyle{remark}
\newtheorem{example}{Example} 
\newtheorem{remark}{Remark} 

\usepackage{fullpage}

\usepackage[utf8]{inputenc} 
\usepackage[T1]{fontenc}    
\usepackage{url}            
\usepackage{booktabs}       
\usepackage{amsfonts}       
\usepackage{nicefrac}       
\usepackage{microtype}      

\usepackage{graphicx}
\graphicspath{{experiments/images/}}

\title{Accelerated Coordinate Descent with Arbitrary Sampling \\ and Best Rates for Minibatches}
 
\author{Filip Hanzely\thanks{King Abdullah University of Science and Technology, Kingdom of Saudi Arabia}   \qquad  Peter Richt\'{a}rik\thanks{School of Mathematics, University of Edinburgh, United Kingdom --- King Abdullah University of Science and Technology, Kingdom of Saudi Arabia --- Moscow Institute of Physics and Technology, Russia}}

\date{May 19, 2018}

\begin{document}

\maketitle

\begin{abstract}
Accelerated coordinate descent is a widely popular optimization algorithm due to its efficiency on large-dimensional problems. It achieves state-of-the-art complexity on an important class of empirical risk minimization problems.  In this paper we design and analyze an accelerated coordinate descent (\texttt{ACD}) method which in each iteration updates a random subset of coordinates according to an arbitrary but fixed probability law, which is a parameter of the method. If all coordinates are updated in each iteration, our method reduces to the classical accelerated gradient descent method \texttt{AGD} of Nesterov. If a single coordinate is updated in each iteration, and we pick probabilities proportional to the square roots of the coordinate-wise Lipschitz constants, our method reduces to the currently fastest coordinate descent method \texttt{NUACDM} of Allen-Zhu, Qu, Richt\'{a}rik and Yuan. 

While mini-batch variants of \texttt{ACD} are more popular and relevant in practice, there is no importance sampling for \texttt{ACD} that outperforms the standard uniform mini-batch sampling. Through insights enabled by our general analysis, we design new importance sampling for mini-batch \texttt{ACD} which significantly outperforms previous state-of-the-art minibatch \texttt{ACD} in practice. We prove a rate that is at most $\cO(\sqrt{\tau})$ times worse than the rate of minibatch \texttt{ACD} with uniform sampling, but can be $\cO(n/\tau)$ times better, where $\tau$ is the minibatch size. Since in modern supervised learning training systems it is standard practice to choose $\tau \ll n$, and often $\tau=\cO(1)$, our method can lead to dramatic speedups. Lastly, we obtain similar results for minibatch nonaccelerated  \texttt{CD} as well, achieving improvements on previous best rates.  
\end{abstract}

\section{Introduction}

Many key problems in machine learning and data science are routinely modeled as optimization problems and solved via optimization algorithms. With the increase of the volume of data used to formulate optimization models, there is a need for new efficient algorithms able to cope with the challenge. Through intensive research and algorithmic innovation during the last 10-15 years, gradient methods  have   become the methods of choice for large-scale optimization problems.

In this paper we consider the optimization problem
\begin{equation}\label{eq:problem}
\min_{x\in \R^n} f(x),
\end{equation}
where  $f$ a smooth and strongly convex function, and the main difficulty  comes from the dimension $n$ being very large (e.g., millions or billions). 
In this regime, {\em coordinate descent (\texttt{CD})} variants of gradient methods are the state of the art. 

The simplest variant of  \texttt{CD} in each iterations updates a single variable of $x$ by taking a one dimensional gradient step along the direction of $i$th unit basis vector $e_i\in \R^n$, which leads to the update rule \begin{equation} \label{eq:CD-intro}\boxed{x^{k+1} = x^k - \alpha_i \nabla_i f(x^k) e_i}\end{equation} where $\nabla_i f(x^k) \eqdef e_i^\top \nabla f(x^k)$ is the $i$th partial derivative and $\alpha_i$ is a suitably chosen stepsize.  The classical smoothness assumption used in the analysis of \texttt{CD} methods~\cite{Nesterov:2010RCDM} is to require the existence of constants $L_i>0$ such that \begin{equation}\label{eq:98g98gf} f(x+ t e_i) \leq f(x) + t \nabla_i f(x)+ \frac{L_i}{2} t^2\end{equation} holds for all $x\in \R^n$, $t\in \R$ and $i\in [n]\eqdef \{1,2,\dots,n\}$. In this setting, one can choose the stepsizes to be $\alpha_i = 1/L_i$. 

There are several rules studied in the literature for  choosing the coordinate $i$ in iteration $k$, including cyclic rule~\cite{luo1992convergence,tseng2001convergence,saha2013nonasymptotic,wright2015coordinate,gurbuzbalaban2017cyclic}, Gauss-Southwell or other greedy rules~\cite{nutini2015coordinate,you2016asynchronous,stich2017approximate}, random (stationary) rule~\cite{Nesterov:2010RCDM, UCDC, PCDM, shalev2014accelerated, lin2014accelerated, APPROX} and adaptive random rules~\cite{csiba15,stich2017safe}. In this work we focus on stationary random rules, which are popular by practitioners and well understood in theory.

\textbf{Updating one coordinate at a time.} The simplest randomized \texttt{CD} method  of the form \eqref{eq:CD-intro} chooses coordinate $i$  in each iteration uniformly at random. If $f$ is $\sigma$--convex\footnote{We say that $f$ is $\sigma$--convex if it is strongly convex with strong convexity modulus $\sigma>0$. That is, if $f(x+h) \geq f(x) + (\nabla f(x))^\top h + \tfrac{\sigma}{2}\|h\|^2$ for all $x,h\in \R^n$, where $\|h\|\eqdef (\sum_i h_i^2)^{1/2}$ is the standard Euclidean norm.}, then this method converges in $(n \max_i  L_i/\sigma) \log (1/\epsilon)$ iterations in expectation. If index $i$ is chosen with probability $p_i \propto L_i$, then the iteration complexity  improves to $(\sum_i L_i/\sigma) \log (1/\epsilon)$. The latter result is always better than the former, and can be up to $n$ times better. These results were established in a seminal paper by Nesterov~\cite{Nesterov:2010RCDM}. The analysis was later generalized to  arbitrary probabilities $p_i>0$ by Richt\'{a}rik and Tak\'{a}\v{c}~\cite{UCDC}, who obtained the complexity \begin{equation} \label{eq:089ff} \left(\max_i \frac{L_i}{p_i \sigma}\right) \log \frac{1}{\epsilon}.\end{equation}  Clearly, \eqref{eq:089ff} includes the previous two results as special cases. Note that the importance sampling $p_i\propto L_i$ minimizes the complexity bound \eqref{eq:089ff} and is therefore in this sense optimal.

\textbf{Mini-batching: updating more coordinates at a time.} In many situations it  is advantageous to update a small {\em subset (mini-batch)} of coordinates in each iteration, which leads to the 
{\em mini-batch \texttt{CD} method} which has  the form
 \begin{equation} \label{eq:Parallel-CD-intro} \boxed{x^{k+1}_i = \begin{cases} x^k_i - \alpha_i \nabla_i f(x^k) & \quad i\in S^k, \\
 x^k_i & \quad i\notin S^k.
 \end{cases}
 }
 \end{equation} 

   For instance, it is often equally easy to fetch information about a small batch of coordinates $S^k$ from memory at the same or comparable  time as it is to fetch information about a single coordinate. If this memory access time is the bottleneck as opposed to computing the actual updates to coordinates $i\in S^k$, then it is more efficient to update all coordinates belonging to the mini-batch $S^k$. Alternatively, in situations where parallel processing is available, one is able to compute the updates to a small batch of coordinates simultaneously, leading to speedups in wall clock time. With this application in mind, mini-batch \texttt{CD} methods are also often called {\em parallel} \texttt{CD} methods \cite{PCDM}.



\section{Arbitrary sampling and  mini-batching}

\textbf{Arbitrary sampling.} Richt\'{a}rik and Tak\'{a}\v{c}~\cite{PCDM} analyzed method \eqref{eq:Parallel-CD-intro}  for {\em uniform samplings} $S^k$, i.e., assuming that $\Prob( i\in S^k) = \Prob(j\in S^k)$  for all $i,j$.  However, the ultimate generalization is captured by the notion of {\em arbitrary sampling}  pioneered by Richt\'{a}rik and Tak\'{a}\v{c} \cite{ NSync}. A {\em sampling} refers to a set-valued random mapping $S$ with values being the subsets of $[n]$. The word {\em arbitrary} refers to the fact that no additional assumptions on the sampling, such as uniformity, are made. This result generalizes the results mentioned above.

For mini-batch \texttt{CD} methods it is useful to assume a more general notion of smoothness parameterized by a positive semidefinite matrix $\mM\in \R^{n\times n}$. We say that $f$ is $\mM$--smooth\footnote{The standard $L$--smoothness condition is obtained in the special case when $\mM = L \mI$, where $\mI$ is the identity matrix in $\R^n$. Note that if $f$ is $\mM$--smooth, then \eqref{eq:98g98gf} holds for $L_i=\mM_{ii}$. Conversely, it is known that if \eqref{eq:98g98gf} holds, then \eqref{eq:M-smooth-intro} holds for $\mM = n \Diag{L_1,L_2,\dots,L_n}$~\cite{Nesterov:2010RCDM}. If $h$ has at most $\omega$ nonzero entries, then this result can be strengthened and \eqref{eq:M-smooth-intro} holds with $\mM = \omega \Diag{L_1,L_2,\dots,L_n}$ \cite[Theorem 8]{PCDM}. In many situations, $\mM$--smoothness is a very natural assumption. For instance, in the context of empirical risk minimization (ERM), which is a key problem in supervised machine learning, $f$ is of the form
$f(x) = \frac{1}{m}\sum_{i=1}^m f_i(\mA_i x) + \frac{\sigma}{2} \|x\|^2,$
where $\mA_i\in \R^{q\times n}$ are data matrices, $f_i:\R^q\to \R$ are loss functions and $\lambda\geq 0$ is a regularization constant.  If $f_i$ is convex and  $\gamma_i$--smooth, then $f$ is $\sigma$--convex and $\mM$--smooth with $\mM = (\frac{1}{m}\sum_i \gamma_i \mA_i^\top \mA_i) + \sigma \mI$ \cite{ESO}. In these situations it is useful to design \texttt{CD} algorithms making full use of the information contained in the data as captured in the smoothness matrix $\mM$.} if 
\begin{equation}\label{eq:M-smooth-intro}f(x+h) \leq f(x) + \nabla f(x)^\top h + \frac{1}{2}h^\top \mM h\end{equation}
for all $x,h\in \R^n$. Given a sampling $S$ and $\mM$--smooth function $f$, let $v=(v_1,\dots,v_n)$ be positive constants satisfying the ESO (expected separable overapproximation) inequality
\begin{equation}\label{eq:v_def}
\mP\circ \mM\preceq \Diag{p_1 v_1, \dots, p_n v_n},
\end{equation}
where $\mP$ is the {\em probability matrix} associated with sampling $S$, defined by $\mP_{ij}\eqdef \Prob(i\in S \; \& \; j\in S)$, $p_i \eqdef \mP_{ii}=\Prob(i\in S)$ and $\circ$ denotes the Hadamard (i.e., elementwise) product of matrices. From now on we define the {\em probability vector} as $p\eqdef (p_1,\dots,p_n)\in \R^n$ and let $v=(v_1,\dots,v_n)\in \R^n$ be the vector of ESO parameters. With this notation, \eqref{eq:v_def} can be equivalently written as $\mP\circ \mM\preceq \Diag{p\circ v}$. We say that $S$ is {\em proper} if $p_i>0$ for all $i$. 

It can be show by combining the results from \cite{NSync} and \cite{ESO} that under the above assumptions, the minibatch \texttt{CD} method \eqref{eq:Parallel-CD-intro} with stepsizes $\alpha_i= 1/v_i$ enjoys the iteration complexity \begin{equation}\label{eq:NSync}\left(\max_{i}\frac{ v_i}{ p_i \sigma}\right) \log \frac{1}{\epsilon} .\end{equation}
Since in situations when $|S^k|=1$ with probability 1 once can choose $v_i=L_i$,  the complexity result \eqref{eq:NSync} generalizes \eqref{eq:089ff}. Inequality~\eqref{eq:v_def} is  standard in minibatch coordinate descent literature. It was studied extensively in \cite{ESO}, and has been used to analyze parallel \texttt{CD} methods \cite{PCDM, NSync, APPROX}, distributed \texttt{CD} methods \cite{Hydra, Hydra2}, accelerated \texttt{CD} methods in \cite{APPROX, Hydra2, ALPHA, SCP}, and dual methods \cite{Quartz, SCP}. 

\textbf{Importance sampling for mini-batches.}  It is easy to see, for instance, that if we do not restrict the class of samplings over which we optimize, then the trivial {\em full sampling} $S^k = [n]$ with probability 1 is optimal. For this sampling, $\mP$ is the matrix of all ones, $p_i=1$ for all $i$, and \eqref{eq:v_def} holds for $v_i=L\eqdef \lambda_{\max}(\mM)$ for all $i$. The mini-batch \texttt{CD} method \eqref{eq:Parallel-CD-intro} reduces  to gradient descent, and the complexity estimate \eqref{eq:NSync} becomes $(L/\sigma) \log (1/\epsilon)$, which is the standard rate of gradient descent. However, typically we are interested in  finding the best sampling from the class of samplings which use a mini-batch of size $\tau$, where $\tau\ll n$. While we have seen that the importance sampling $p_i=L_i/\sum_j L_j$ is optimal for $\tau=1$, in the mini-batch case $\tau>1$  the problem of determining a sampling which minimizes the bound \eqref{eq:NSync} is much more difficult. For instance, \cite{NSync} consider a certain parametric family of samplings where the problem of finding the best sampling from this family reduces to a linear program. 

Surprisingly, and in contrast to the situation in the $\tau=1$ case where an optimal sampling is known and is in general non-uniform, there is no mini-batch sampling that is guaranteed to outperform $\tau$--nice sampling. We say that $S$ is $\tau$--nice if it samples uniformly from among all subsets of $[n]$ of cardinality $\tau$. The probability matrix of this sampling is given by 
\[\mP = \frac{\tau}{n}\left((1-\beta) \mI + \beta \mE \right),\]
where $\beta = \tfrac{\tau-1}{n-1}$ (assume $n>1$) and $\mE$ is the matrix of all ones, and $p_i = \tfrac{\tau}{n}$ \cite{ESO}. It follows that the ESO inequality \eqref{eq:v_def} holds for $v_i = (1-\beta) \mM_{ii} + \beta L.$ By plugging into  \eqref{eq:NSync}, we get the iteration complexity
\begin{equation}\label{eq:tau-nice-rate} \frac{n}{\tau}\left( \frac{ (1-\beta) \max_i \mM_{ii} + \beta L}{\sigma}  \right) \log \frac{1}{\epsilon}.\end{equation}

This rate interpolates between the rate of \texttt{CD} with uniform probabilities (for $\tau=1$) and the rate of gradient descent (for $\tau=n$).


\begin{table}
\centering
\begin{tabular}{|c|c|c|c|c|}
\hline
& \texttt{CD}  &   \texttt{ACD}  \\
\hline
$\tau=1$, $p_i >0 $ & $\left(\max_i \frac{L_i}{p_i \sigma}\right) \log \frac{1}{\epsilon}$~\cite{UCDC} & $ \sqrt{ \max_i \frac{L_i}{p_i^2 \sigma} }   \log \frac{1}{\epsilon}$ {\bf [this paper]}\\
\hline
$\tau=1$, best $p_i$ & $\frac{\sum_i L_i}{\sigma }\log \frac{1}{\epsilon}$; \quad  $p_i\propto L_i$~\cite{Nesterov:2010RCDM} &$\frac{\sum_i \sqrt{L_i}}{\sqrt{\sigma}}  \log \frac{1}{\epsilon}$;  \quad $p_i\propto \sqrt{L_i}$~\cite{allen2016even} \\
\hline
\hline
arbitrary sampling & $\left(\max_i \frac{v_i}{p_i \sigma}\right) \log \frac{1}{\epsilon}$~\cite{NSync} & $ \sqrt{ \max_i \frac{v_i}{p_i^2 \sigma} }   \log \frac{1}{\epsilon}$ {\bf [this paper]}\\
\hline
\end{tabular}
\caption{Complexity results for \texttt{CD} and \texttt{ACD} methods for $\sigma$--convex functions.}
\label{tab:main}
\end{table}

\section{Contributions}

For {\em accelerated coordinate descent (\texttt{ACD})} without mini-batching (i.e., when $\tau=1$), the currently best known iteration complexity result, due to Allen-Zhu et al~\cite{allen2016even}, is
\begin{equation}\label{eq:ZAZ}\cO\left(\frac{\sum_i \sqrt{L_i}}{\sqrt{\sigma}}  \log \frac{1}{\epsilon}\right).\end{equation}
 The probabilities used in the algorithm are proportional to the square roots of the coordinate-wise Lipschitz constants: $p_i \propto \sqrt{L_i}$. This is the first \texttt{CD} method with a complexity guarantee which does not explicitly depend on the dimension $n$, and is an improvement on the now-classical result of Nesterov~\cite{Nesterov:2010RCDM} giving the complexity
\[\cO\left(\sqrt{ \frac{n \sum_i L_i }{ \sigma}}  \log \frac{1}{\epsilon} \right).\]
The rate \eqref{eq:ZAZ} is always better than this, and can be up to $\sqrt{n}$ times better if the distribution of $L_i$ is extremely non-uniform. Unlike in the non-accelerated case described in the previous section, there is no complexity result for \texttt{ACD} with general probabilities such as \eqref{eq:089ff}, or with an arbitrary sampling such as \eqref{eq:NSync}. In fact, an \texttt{ACD} method was not even designed in such settings, despite a significant recent development in accelerated coordinate descent methods~\cite{nesterov2012efficiency,Lee2013,lin2014accelerated,ALPHA,allen2016even}.

Our key contributions are:

 {\bf $\diamond$ \texttt{ACD} with arbitrary sampling.} We design an \texttt{ACD} method which is able to operate with an {\em arbitrary sampling} of subsets of coordinates. We describe our method in Section~\ref{sec:ACD}.
 
{\bf $\diamond$  Iteration complexity.} We prove (see Theorem~\ref{th:acd}) that the iteration complexity of \texttt{ACD} is 
\begin{equation}\label{eq:ug98sg98s}\cO\left( \sqrt{ \max_i \frac{v_i}{p_i^2 \sigma} }   \log \frac{1}{\epsilon}\right),\end{equation}
where $v_i$ are ESO parameters given by \eqref{eq:v_def} and $p_i>0$ is the probability that coordinate $i$ belongs to the sampled set $S^k$: $p_i\eqdef \Prob(i\in S^k)$. The result of Allen-Zhu et al.\ \eqref{eq:ZAZ} (\texttt{NUACDM}) can be recovered as a special case of  \eqref{eq:ug98sg98s} by focusing on samplings defined by $S^k=\{i\}$ with probability $p_i \propto \sqrt{L_i} $ (recall that in this case $v_i=L_i$). When $S^k=[n]$ with probability 1, then our method reduces to accelerated gradient descent (\texttt{AGD})~\cite{nesterov1983method,NesterovBook}, and since $p_i=1$  and $v_i=L$ (the Lipschitz constant of $\nabla f$) for all $i$,  \eqref{eq:ug98sg98s} reduces to the standard complexity of \texttt{AGD}:
$\cO(\sqrt{L/\sigma} \log (1/\epsilon)).$

 {\bf $\diamond$  Weighted strong convexity.} In fact, we prove a more general result than \eqref{eq:ug98sg98s} in which we allow the strong convexity of $f$ to be measured in a weighted Euclidean norm with weights $v_i/p_i^2$. In situations when $f$ is naturally strongly convex with respect to a weighted norm, this more general result will typically lead to a better complexity result than \eqref{eq:ug98sg98s}, which is fine-tuned for standard strong convexity.
There are applications when $f$ is naturally a strongly convex with respect to some weighted norm~\cite{allen2016even}. 

{\bf $\diamond$ Mini-batch methods.} We design several {\em new} importance samplings for mini-batches, calculate the associated complexity results, and show through experiments that they significantly outperform the standard uniform samplings used in practice and constitute the state of the art. Our importance sampling leads to rates which are provably within a small factor from the best known rates, but can lead to an improvement by a factor of $\cO(n)$. We are the first to establish such a result, both for \texttt{CD} (Appendix~\ref{sec:cd_imp}) and \texttt{ACD} (Section~\ref{sec:import}).

The key complexity results obtained in this paper are summarized and compared to prior results in Table~\ref{tab:main}.

\section{The algorithm} \label{sec:ACD}

The accelerated coordinate descent method (\texttt{ACD}) we propose is formalized as Algorithm~\ref{alg:acd}. As mentioned before, we will analyze our method under a more general strong convexity assumption.

\begin{assumption}\label{ass:sc}
$f$ is $\sigma_w$--convex with respect to the  $\|\cdot \|_{w}$ norm. That is, 
\begin{equation}\label{eq:sc}
f(x+h)\geq f(x) +\langle \nabla f(x),h\rangle +\frac{\sigma_w}{2}\|h\|_w^2,
\end{equation}
for all $x,h\in \R^n$, where $\sigma_w>0$.
\end{assumption}

Note that if $f$ is $\sigma$--convex in the standard sense (i.e., for $w=(1,\dots,1)$), then  $f$ is $\sigma_w$--convex for any $w>0$ with $\sigma_w = \min_i \frac{\sigma}{w_i}.$

\begin{algorithm}[H] 
\textbf{Input: }{i.i.d.\ proper samplings $S^k\sim \cD$; $v,w\in \R^n_{++}$; $\sigma_w>0$; stepsize parameters $\eta,\theta>0$}\\
\textbf{Initialize: }{Initial iterate $y^0=z^0\in \R^n $ }\\
\For {$k= 0,1\dots $} {
\begin{align}
&x^{k+1}=(1-\theta)y^{k}+\theta z^{k} \label{eq:x_update_acd}\\
&\text{Get } S^k \sim \cD
\\
 & y^{k+1}=x^{k+1}-\sum_{i\in S^k} \frac{1}{v_i} \nabla_if(x^{k+1}) e_i\label{eq:y_update}\\
  & z^{k+1}=\frac{1}{1+\eta\sigma_w}\left(z^k+\eta\sigma_w x^{k+1}-\sum_{i\in S^k}\frac{\eta}{p_i w_i } \nabla_i f(x^{k+1}) e_i \right) \label{eq:z_update}
 \end{align}
 }
\caption{\texttt{ACD} (Accelerated coordinate descent with arbitrary sampling)}
\label{alg:acd}
\end{algorithm}

Using the tricks developed in \cite{Lee2013, APPROX, lin2014accelerated}, Algorithm~\ref{alg:acd} can be implemented so that only $|S^k|$ coordinates are updated in each iteration. We are now ready derive a convergence rate of \texttt{ACD}.

\begin{theorem}[Convergence of \texttt{ACD}]\label{th:acd} Let $S^k$ be i.i.d.\ proper (but otherwise arbitrary) samplings. Let $\mP$ be the associated probability matrix and  $p_i\eqdef \Prob(i\in S^k)$.   Assume $f$ is $\mM$--smooth (see \eqref{eq:M-smooth-intro}) and let $v$ be ESO parameters satisfying \eqref{eq:v_def}.  Further, assume that $f$ is $\sigma_w$--convex (with $\sigma_w>0$) for
 \begin{equation}
w_i\eqdef \frac{v_i}{p_i^2}, \qquad i=1,2,\dots,n,\label{eq:w_def}
\end{equation}
with respect to the weighted Euclidean norm $\|\cdot\|_w$ (i.e., we enforce Assumption~\ref{ass:sc}).  Then  \begin{equation} \label{eq:998dgff}\sigma_w \leq \frac{\mM_{ii} p_i^2}{v_i} \leq p_i^2 
\leq 1, \qquad i=1,2,\dots n.\end{equation}
In particular, if $f$ is $\sigma$--convex with respect to the standard Euclidean norm, then  we can choose \begin{equation}\label{eq:8h8hs8s}\sigma_w = \min_i \frac{p_i^2 \sigma}{v_i}.\end{equation} 
Finally, if we choose
\begin{equation}\label{eq:tau_def_acd}
\theta\eqdef\frac{\sqrt{\sigma_w^2+4\sigma_w}-\sigma_w}{2}=\frac{2\sigma_w}{\sqrt{\sigma_w^2+4\sigma_w}+\sigma_w} \geq 0.618 \sqrt{\sigma_\omega}
\end{equation}
and 
$
\eta\eqdef \frac{1}{\theta},
$
then the random iterates of \texttt{ACD} satisfy
\begin{equation}\label{eq:recur}
\E{P^{k}}
\leq
(1-\theta)^kP^0,
\end{equation}
where $P^k\eqdef \frac{1}{\theta^2}\left( f(y^k)-f(x^*)\right)+\frac{1}{2(1-\theta)}\|z^k-x^* \|_w^2$ and $x^*$ is the optimal solution of \eqref{eq:problem}. 
\end{theorem}

 Noting that $1/0.618\leq 1.619$, as an immediate consequence of \eqref{eq:recur} and  \eqref{eq:tau_def_acd} we get bound
\begin{equation} \label{eq:ius98g9skkk} k \geq \frac{1.619}{\sqrt{\sigma_w}} \log \frac{1}{\epsilon} \quad \Rightarrow \quad \E{P^k} \leq \epsilon P^0.\end{equation}
 If $f$ is $\sigma$--convex, then by plugging \eqref{eq:8h8hs8s} into \eqref{eq:ius98g9skkk} we obtain the iteration complexity bound
\begin{equation}\label{eq:98g98df} 1.619 \cdot \sqrt{\max_i \frac{v_i}{p_i^2 \sigma}} \log \frac{1}{\epsilon}.\end{equation}
Complexity \eqref{eq:98g98df} is our key result (also  mentioned in \eqref{eq:ug98sg98s} and Table~\ref{tab:main}).

\section{Importance sampling for mini-batches \label{sec:import}}

Let $\tau\eqdef \E{|S^k|}$ be the expected mini-batch size. The next theorem provides an insightful lower bound for the complexity of \texttt{ACD} we established, one independent of  $p$ and $v$. 

\begin{theorem}[Limits of mini-batch performance] \label{thm:LB} Let the assumptions of Theorem~\ref{th:acd}  be satisfied and let $f$ be $\sigma$--convex. Then the dominant term in the rate~\eqref{eq:98g98df} of \texttt{ACD} admits the lower bound
\begin{equation}\label{eq:ineq-minibatch-speedup} \sqrt{\max_i \frac{v_i}{p_i^2 \sigma}}  \geq \frac{\sum_{i} \sqrt{\mM_{ii}}}{\tau \sqrt{\sigma}}. \end{equation}
\end{theorem}
Note that for $\tau=1$ we have $\mM_{ii}=v_i=L_i$, and the lower bound is achieved by using the importance sampling $p_i\propto \sqrt{L_i}$.
Hence, this bound gives a  limit on how much speedup, compared to the best known complexity in the  $\tau=1$ case, we can hope for  as we increase $\tau$. The bound says we can not hope for better than linear speedup in the mini-batch size.  An analogous result (obtained by removing all the squares and square roots in \eqref{eq:ineq-minibatch-speedup}) was established in \cite{NSync} for  \texttt{CD}.


In what follows, it will be useful to write the complexity result \eqref{eq:98g98df} in a new form by considering a specific choice of the ESO vector $v$.

\begin{lemma}\label{thm:special-ESO-result} Choose any proper sampling $S$ and let $\mP$ be its probability matrix and $p$ its probability vector. Let $c(S,\mM) \eqdef \lambda_{\max}(\mP'\circ \mM')$, where $\mP' \eqdef \mD^{-1/2} \mP \mD^{-1/2}$, $\mM' \eqdef \mD^{-1}\mM \mD^{-1}$ and $\mD \eqdef \Diag{p}$. Then the vector $v$ defined by $v_i = c(S,\mM) p_i^2$ satisfies  the ESO inequality \eqref{eq:v_def} and the total complexity~\eqref{eq:98g98df} becomes
\begin{equation}\label{eq:98g98df2} 
1.619 \cdot \frac{\sqrt{c(S,\mM)}}{\sqrt{\sigma}} \log \frac{1}{\epsilon}.
\end{equation}
\end{lemma}

Since $\tfrac{1}{n}\trace{\mP'\circ \mM'}\leq c(S,\mM) \leq \trace{\mP'\circ \mM'}$ and $\trace{\mP'\circ \mM'}=\sum_i \mP'_{ii} \mM'_{ii} = \sum_i  \mM'_{ii} =\sum_i \mM_{ii}/p_i^2$, we get the bounds:
\begin{equation}\label{eq:nbisg8dd}   \sqrt{\frac{1}{n}\sum_i \frac{\mM_{ii}}{p_i^2\sigma}} \log \frac{1}{\epsilon}\leq  \sqrt{\frac{c(S,\mM)}{\sigma}} \log \frac{1}{\epsilon}\leq   \sqrt{\sum_i \frac{\mM_{ii}}{p_i^2\sigma}} \log \frac{1}{\epsilon}.\end{equation}

\subsection{Sampling 1: standard uniform minibatch samlpling ($\tau$--nice sampling)\label{sec:sam1}} Let $S_1$ be the $\tau$-nice sampling. It can be shown (see Lemma~\ref{lem:tau-nice-2nd-derivation}) that $c(S_1,\mM) \leq\frac{n^2}{\tau^2} ((1-\beta)\max_i \mM_{ii} + \beta L)$, and  hence the iteration complexity \eqref{eq:98g98df} becomes \begin{equation}\label{eq:ug98sg98s-tau-nice}\cO\left( \frac{n}{\tau}\sqrt{ \frac{(1-\beta) \max_i \mM_{ii} + \beta L}{ \sigma} }   \log \frac{1}{\epsilon}\right).\end{equation}
This result interpolates between \texttt{ACD} with uniform probabilities (for $\tau=1$) and accelerated gradient descent (for $\tau=n$). Note that the rate \eqref{eq:ug98sg98s-tau-nice} is a strict improvement on the \texttt{CD} rate  \eqref{eq:tau-nice-rate}.

\subsection{Sampling 2: importance sampling for minibatches} \label{sec:sam2}
Consider now the sampling $S_2$ which  includes every $i \in [n]$ in $S_2$, independently, with probability \[p_i = \tau \frac{\sqrt{\mM_{ii}}}{\sum_j \sqrt{\mM_{jj}}}.\] This sampling was not considered in the literature before. Note that $\E{|S_2|}=\sum_i p_i = \tau$. For this sampling, bounds \eqref{eq:nbisg8dd}  become:
\begin{equation}\label{eq:LB--x}  \frac{\sum_{i} \sqrt{\mM_{ii}}}{\tau \sqrt{\sigma}}\log \frac{1}{\epsilon}\leq  \sqrt{\frac{c(S,\mM)}{\sigma}} \log \frac{1}{\epsilon}\leq   \frac{\sqrt{n}\sum_{i} \sqrt{\mM_{ii}}}{\tau \sqrt{\sigma}} \log \frac{1}{\epsilon}.\end{equation}
Clearly, with this sampling we obtain an \texttt{ACD} method with complexity within a $\sqrt{n}$ factor from the lower bound established in Theorem~\ref{thm:LB}. For $\tau=1$ we have  $\mP'=\mI$ and hence $c(S,\mM)=\lambda_{\max}(\mI\circ \mM') = \lambda_{\max}(\Diag{\mM'}) = \max_i \mM_{ii}/p_i^2 = (\sum_j \sqrt{\mM_{jj}})^2$. Thus, the rate of \texttt{ACD} achieves the lower bound in \eqref{eq:LB--x} (see also \eqref{eq:ZAZ}) and we recover the best current rate of \texttt{ACD} in the $\tau=1$ case, established in \cite{allen2016even}. However, the sampling has an important limitation: it can be used for $\tau\leq \sum_j \sqrt{\mM_{jj}}/\max_i \mM_{ii}$ only  as otherwise the probabilities $p_i$ exceed 1.

\subsection{Sampling 3: another importance sampling for minibatches} \label{sec:sam3}
Now consider sampling $S_3$ which includes each coordinate $i$ within $S_3$ independently, with probability $p_i$ satisfying the relation $p_i^2/\mM_{ii} \propto 1-p_i$. This is equivalent to setting
\begin{equation}\label{eq:p_imp_def_acc1}
p_{i}\eqdef\frac{2 \mM_{ii}}{\sqrt{\mM_{ii}^2+2\mM_{ii}\delta}+\mM_{ii}},
\end{equation}
where $\delta$ is a scalar for which $\sum_i p_i=\tau$. This sampling was not considered in the literature before. Probability vector $p$ was chosen as~\eqref{eq:p_imp_def_acc1} for two reasons: i) $p_i\leq 1$ for all $i$, and therefore the sampling can be used for all $\tau$ in contrast to $S_1$, and ii) we can prove Theorem~\ref{thm:comparison}. 


Let $c_1\eqdef c(S_1,\mM)$ and $c_3\eqdef c(S_3,\mM)$. In light of \eqref{eq:98g98df2}, Theorem~\ref{thm:comparison} compares $S_1$ and $S_3$ and says that \texttt{ACD} with $S_3$ has at most $\cO(\sqrt{\tau})$ times worse rate compared to \texttt{ACD} with $S_1$, but has the capacity to be $\cO(n/\tau)$ times better. We prove in Appendix~\ref{sec:cd_imp} a similar theorem for \texttt{CD}. We stress that, despite some advances in the development of importance samplings for minibatch methods~\cite{NSync, csiba2018importance}, $S_1$ was until now the state-of-the-art in theory for \texttt{CD}. We are the first to give a provably better rate in the sense of Theorem~\ref{thm:nonacc_comp}. The numerical experiments show that $S_3$ consistently outperforms $S_1$, and often dramatically so.

\begin{theorem}\label{thm:comparison} The leading complexity terms $c_1$ and $c_3$ of \texttt{ACD} (Algorithm~\eqref{eq:Parallel-CD-intro}) with samplings $S_1$, and $S_3$, respectively, defined in Lemma~\ref{thm:special-ESO-result}, compare as follows:
 \begin{equation}\label{eq:thm_comparison}
  c_3 \leq 2\frac{(2n-\tau)(n\tau+n-\tau)}{(n-\tau)^2}   c_1 =\cO(\tau) c_1.
 \end{equation}
Moreover, there exists $\mM$ where $c_3\leq \cO(\tfrac{\tau^2}{n^2})c_1$. 
\end{theorem}

In real world applications,  minibatch size $\tau$ is limited by  hardware and in typical situations one has  $\tau \ll n$, oftentimes $\tau=\cO(1)$. The importance of Theorem~\ref{thm:comparison} is best understood from this  perspective.

\begin{table}[t]
\centering
\begin{tabular}{|c|c|c|c|c|}
\hline
 Lower bound & $S_1$: $p_i = \tfrac{\tau}{n}$  & $S_2: \tfrac{p_i^2}{\mM_{ii}}\propto 1$ & $S_3: \tfrac{p_i^2}{\mM_{ii}} \propto 1-p_i$ \\
\hline
 $\tfrac{\sum_i \sqrt{\mM_{ii}}}{\tau \sqrt{\sigma}}$  & $  \tfrac{n\sqrt{(1-\beta) \max_i \mM_{ii} + \beta L}}{ \tau\sqrt{\sigma}}   $ &  $\tfrac{\gamma\sum_i \sqrt{\mM_{ii}}}{\tau \sqrt{\sigma}}$ & $\omega  \tfrac{n\sqrt{(1-\beta) \max_i \mM_{ii} + \beta L}}{ \tau\sqrt{\sigma}} $ \\
\hline
\eqref{eq:ineq-minibatch-speedup} & 
\begin{tabular}{c} = uniform \texttt{ACD} for $\tau=1$ \\
= \texttt{AGD}  for $\tau=n$ 
\end{tabular} 
& 
\begin{tabular}{c} $\leq \sqrt{n} \times $  lower bound \\ $\bullet$ $\tau\leq \tfrac{\sum_j \sqrt{\mM_{jj}}}{\max_i \mM_{ii}}$ \end{tabular} & 
\begin{tabular}{c} $\bullet$ fastest in practice \\ $\bullet$ any $\tau$ allowed \end{tabular} 
 \\
\hline
\end{tabular}
\caption{New complexity results for  \texttt{ACD}  with mini-batch size $\tau=\E{|S^k|}$ and various samplings (we suppress $\log (1/\epsilon)$ factors in all expressions).  Constants: $\sigma=$ strong convexity constant of $f$, $L=\lambda_{\max}(\mM)$, $\beta = (\tau-1)/(n-1)$, $1\leq \gamma \leq \sqrt{n}$,  and $\omega \leq \cO(\sqrt{\tau})$ ($\omega$ can be as small as $\cO(\tau/n)$).}
\label{tab:main2}
\end{table}

\section{Experiments} \label{sec:exp}

We perform extensive numerical experiments to justify that minibatch \texttt{ACD} with importance sampling works well in practice. Here we present a few selected experiment only; more can be found in Appendix~\ref{sec:extra_exp}. 

In most of plots we compare of both accelerated and non-accelerated \texttt{CD} with all samplings $S_1,S_2,S_3$ introduced in Sections~\ref{sec:sam1}, \ref{sec:sam2} and \ref{sec:sam3} respectively. We refer to \texttt{ACD} with sampling $S_3$ as \texttt{AN} (\texttt{A}ccelerated \texttt{N}onuniform), \texttt{ACD} with sampling $S_1$ as \texttt{AU}, \texttt{ACD} with sampling $S_2$ as \texttt{AN2}, \texttt{CD} with sampling $S_3$ as \texttt{NN}, \texttt{CD} with sampling $S_1$as \texttt{NU} and \texttt{CD} with sampling $S_2$ as \texttt{NN2}. We compare the methods for various choices of the expected minibatch sizes $\tau$ and on several problems.

In Figure~\ref{fig:logreg_big}, we report on a logistic regression problem with a few selected LibSVM~\cite{chang2011libsvm} datasets.  For larger datasets, pre-computing both strong convexity parameter $\sigma$ and  $v$ may be expensive (however, recall that for $v$ we need to tune only one scalar). Therefore, we choose ESO parameters $v$ from Lemma~\ref{thm:special-ESO-result}, while estimating the smoothness matrix as $10\times$ its diagonal. An estimate of the strong convexity $\sigma$ for acceleration was chosen to be the minimal diagonal element of the smoothness matrix. We provide a formal formulation of the logistic regression problem, along with  more experiments applied to further datasets in Appendix~\ref{exp:logreg}, where we choose $v$ and $\sigma$ in full accord with the theory.

\begin{figure}[t]
\centering
\begin{minipage}{0.25\textwidth}
  \centering
\includegraphics[width =  \textwidth ]{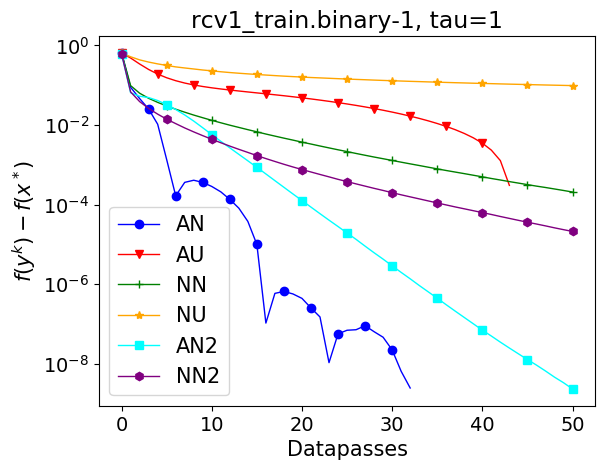}
\end{minipage}%
\begin{minipage}{0.25\textwidth}
  \centering
\includegraphics[width =  \textwidth ]{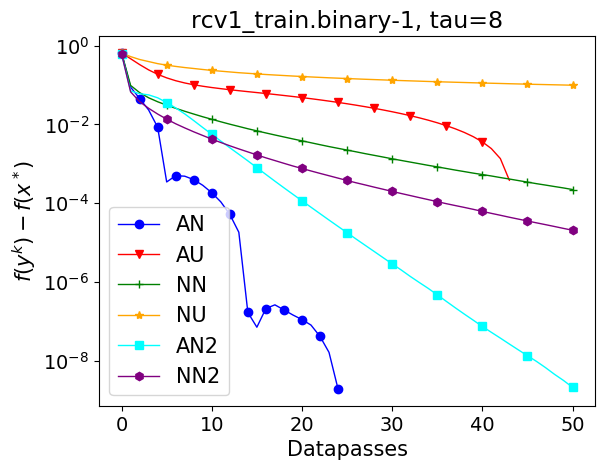}
\end{minipage}%
\begin{minipage}{0.25\textwidth}
  \centering
\includegraphics[width =  \textwidth ]{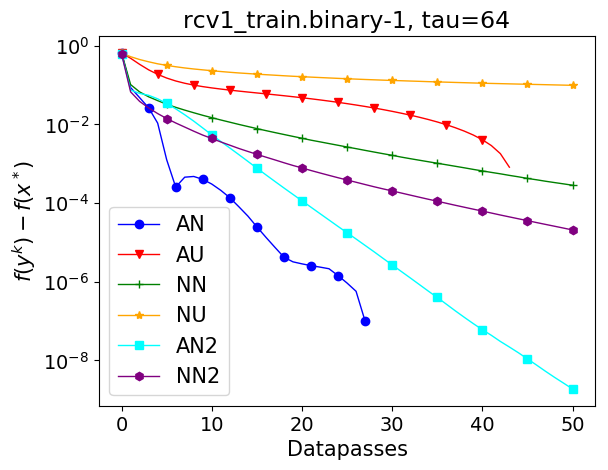}
\end{minipage}%
\begin{minipage}{0.25\textwidth}
  \centering
\includegraphics[width =  \textwidth ]{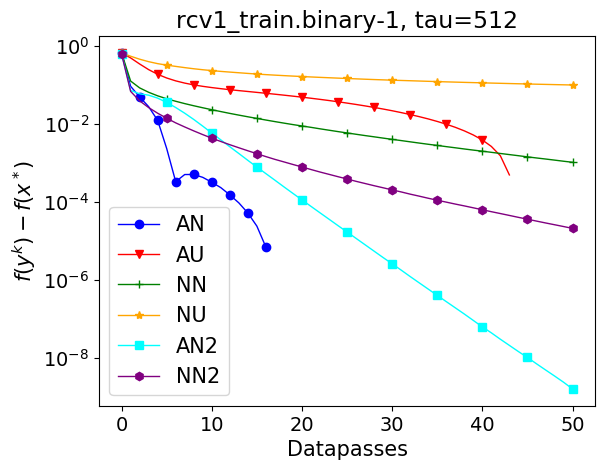}
\end{minipage}%
\\
\begin{minipage}{0.25\textwidth}
  \centering
\includegraphics[width =  \textwidth ]{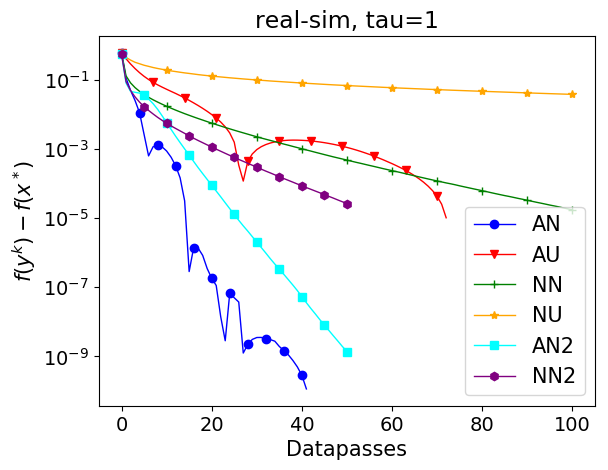}
\end{minipage}%
\begin{minipage}{0.25\textwidth}
  \centering
\includegraphics[width =  \textwidth ]{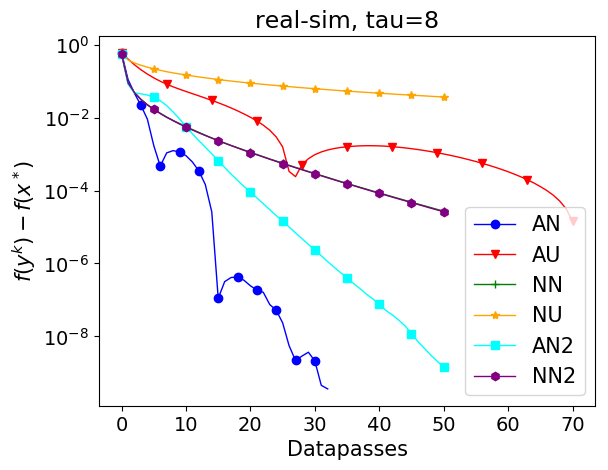}
\end{minipage}%
\begin{minipage}{0.25\textwidth}
  \centering
\includegraphics[width =  \textwidth ]{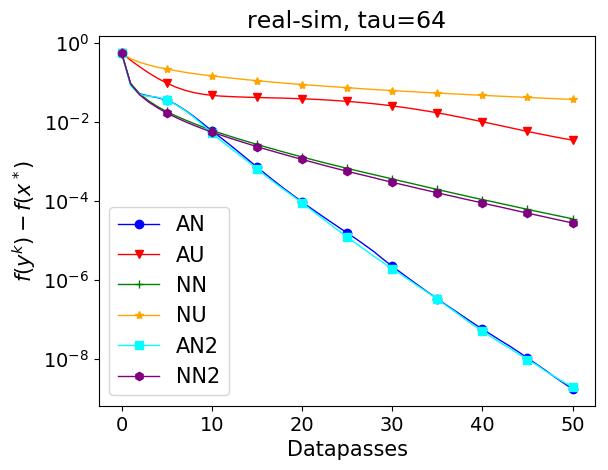}
\end{minipage}%
\begin{minipage}{0.25\textwidth}
  \centering
\includegraphics[width =  \textwidth ]{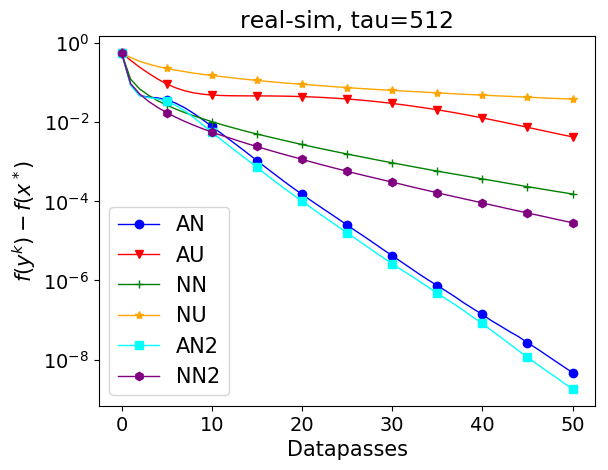}
\end{minipage}%
\\
\begin{minipage}{0.25\textwidth}
  \centering
\includegraphics[width =  \textwidth ]{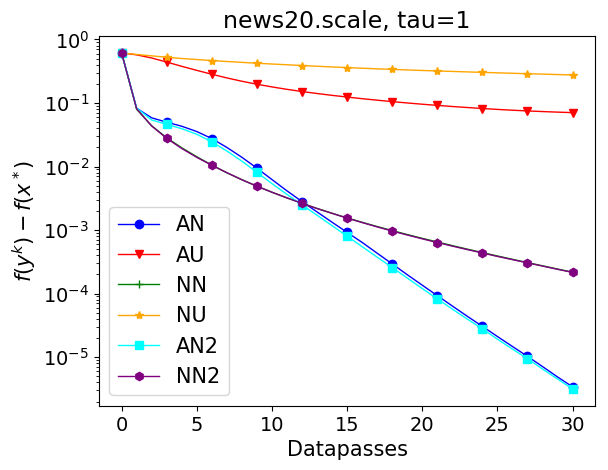}
\end{minipage}%
\begin{minipage}{0.25\textwidth}
  \centering
\includegraphics[width =  \textwidth ]{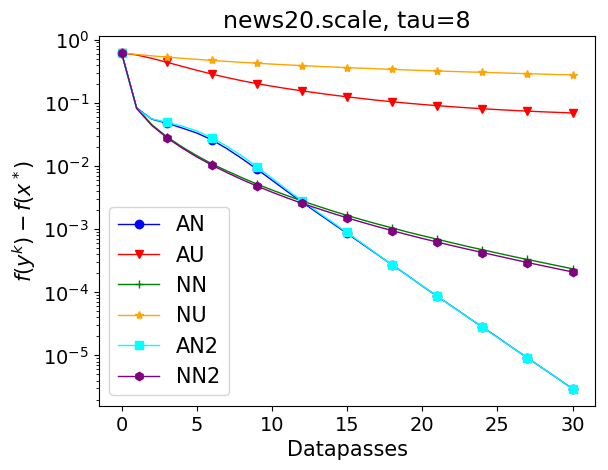}
\end{minipage}%
\begin{minipage}{0.25\textwidth}
  \centering
\includegraphics[width =  \textwidth ]{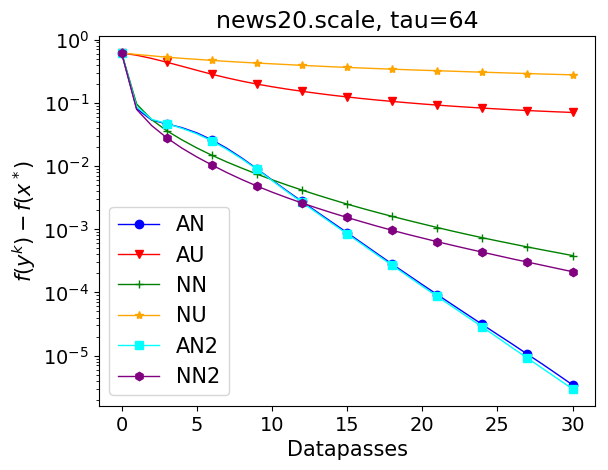}
\end{minipage}%
\begin{minipage}{0.25\textwidth}
  \centering
\includegraphics[width =  \textwidth ]{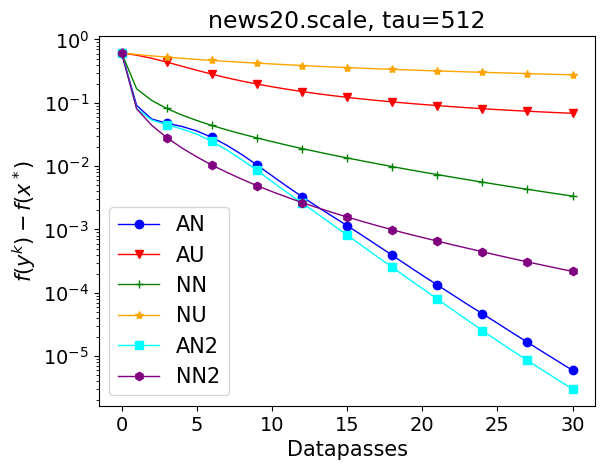}
\end{minipage}%
\caption{Six variants of coordinate descent (\texttt{AN}, \texttt{AU}, \texttt{NN}, \texttt{NU}, \texttt{AN2} and \texttt{AU2})  applied to a logistic regression problem, with minibatch sizes $\tau=1, 8, 64$ and $512$.}\label{fig:logreg_big}
\end{figure}

Coordinate descent methods which allow for separable proximal operator were proven to be efficient to solve ERM problem, when applied on dual~\cite{shalev2011stochastic,SDCA,shalev2014accelerated,zhao2015stochastic}.
Although we do not develop proximal methods in this paper, we empirically demonstrate that \texttt{ACD} allows for this extension as well. As a specific problem to solve, we choose dual of SVM with hinge loss. The results and a detailed description of the experiment are presented in Appendix~\ref{sec:SVM}, and are indeed in favour of \texttt{ACD} with importance sampling. Therefore, \texttt{ACD} is not only suitable for \emph{big dimensional} problems, it can handle the \emph{big data} setting as well.

Finally, in Appendix~\ref{sec:artif} we present several synthetic examples in order to shed more light on  acceleration and importance sampling, and to see how its performance depends on the  data. We also study how minibatch size influences the convergence rate. All the experimental results clearly show that  acceleration, importance sampling and minibatching have a significant impact on practical performance of \texttt{CD} methods. Moreover, the difference in the performance of samplings $S_2$ and $S_3$ is negligible, and therefore we recommend using $S_3$, as it is not limited by the bound on expected minibatch sice $\tau$


 \bibliographystyle{plain} 
\bibliography{literature}

\clearpage
\appendix
\part*{Appendix}

\section{Proof of Theorem~\ref{th:acd}}
Before starting the proof, we mention that the proof technique we use is inspired by~\cite{allen2014linear,allen2016even}, which takes the advantage of the coupling of gradient descent with mirror descent, resulting in a relatively simple proof.

\subsection{Proof of inequality \eqref{eq:998dgff}} \label{subsec:sigma_w<=1}

By comparing \eqref{eq:sc} and \eqref{eq:M-smooth-intro} for $h=e_i$, we get  $\sigma_w w_i \leq \mM_{ii}$, and the first inequality in \eqref{eq:998dgff} follows. Using \eqref{eq:v_def} it follows that  $e_i^\top (\mP \circ \mM) e_i \preceq e_i^\top \Diag{p\circ v} e_i$, which in turn implies $\mM_{ii} \leq v_i$ and the second inequality in \eqref{eq:998dgff} follows.

\subsection{Descent lemma}

The following lemma is a consequence of $\mM$--smoothness of $f$, and ESO inequality~\eqref{eq:v_def}.

\begin{lemma} Under the assumptions of Theorem~\ref{th:acd}, for all $k\geq 0$ we have the bound
\begin{equation}
\label{eq:eso_inq}
f(x^{k+1})-\E{f(y^{k+1})\,|\, x^{k+1}} \geq  \frac12 \| \nabla f(x^{k+1})\|^2_{v^{-1}\circ p}.
\end{equation}
\end{lemma}
\begin{proof}
We have
\begin{eqnarray*}
\E{f(y^{k+1})} &\stackrel{\eqref{eq:y_update}}{=}&\E{f\left(x^{k+1}-\sum_{i\in S^k} \frac{1}{v_i} \nabla_if(x^{k+1}\right) e_i) }
\\
&\stackrel{\eqref{eq:M-smooth-intro}}{\leq}&
f(x^{k+1}) - \|\nabla f(x^{k+1}) \|^2_{v^{-1}\circ p}+ \frac12\E{\left\| \sum_{i\in S^k} \frac{1}{v_i} \nabla_if(x^{k+1}) e_i \right\|^2_{\mM}}
\\
&\stackrel{\eqref{eq:v_def}}{\leq}&
f(x^{k+1}) - \|\nabla f(x^{k+1}) \|^2_{v^{-1}\circ p}+ \frac12 \left\| \nabla f(x^{k+1})  \right\|^2_{v^{-1}\circ p}.
\end{eqnarray*}
\end{proof}

\subsection{Key technical inequality}

We first establish a  lemma which will play a key part in the analysis.

\begin{lemma}\label{l:mirror_acd}
For every $u$ we have
\begin{eqnarray*}
&&
\eta \sum_{i\in S^k} \left\langle \frac{1}{p_i}\nabla_i f(x^{k+1}) e_i, z^{k+1}-u  \right\rangle -\frac{\eta\sigma_w}{2}\|x^{k+1}-u \|^2_w \\
&&
\qquad \qquad
\leq -\frac12\| z^k-z^{k+1}\|^2_w+\frac{1}{2}\|z^k-u \|_w^2-\frac{1+\eta \sigma_w}{2}\| z^{k+1}-u\|^2_{w}.
\end{eqnarray*}
\end{lemma}

\begin{proof}
The proof is a direct generalization of the proof of analogous lemma in \cite{allen2016even}. We include it for completeness.
Notice that \eqref{eq:z_update} is equivalent to
\[
z^{k+1}=\argmin_z h^k(z)\eqdef \argmin_z \frac{1}{2}\| z-z^k\|^2_{w}+\eta \sum_{i\in S^k} \langle \frac{1}{p_i}\nabla_i f(x^{k+1}), z \rangle   +\frac{\eta \sigma_w}{2}\|z-x^{k+1} \|_{w}^2.
\]
Therefore, we have for every $u$
\begin{eqnarray}
\nonumber
0&=&\langle \nabla h^k(z^{k+1}),z^{k+1}-u \rangle_{w} \\
&=&
\langle z^{k+1}-z^k, z^{k+1}-u\rangle_{w} +\eta \sum_{i\in S^k} \langle \frac{1}{p_i}\nabla_i f(x^{k+1}), z^{k+1}-u  \rangle 
\\
&& \qquad \qquad
+\eta \sigma_w \langle z^{k+1}-x^{k+1}, z^{k+1}-u\rangle_{w}. \label{eq:zk_plus_1_optimal}
\end{eqnarray}
Next, by generalized Pythagorean theorem we have
\begin{equation}\label{eq:pyt_z}
\langle z^{k+1}-z^k,z^{k+1}-u \rangle_w =\frac12 \|z^k-z^{k+1} \|^2_w-\frac12 \|z^k-u\|^2_w+\frac12 \|u-z^{k+1} \|^2_w
\end{equation}
and
\begin{equation}\label{eq:pyt_x}
\langle z^{k+1}-x^{k+1},z^{k+1}-u \rangle_w =\frac12 \|x^{k+1}-z^{k+1} \|^2_w-\frac12 \|x^{k+1}-u\|^2_w+\frac12 \|u-z^{k+1} \|^2_w.
\end{equation}
It remains to put~\eqref{eq:pyt_z} and~\eqref{eq:pyt_x} into~\eqref{eq:zk_plus_1_optimal}.
\end{proof}

\subsection{Proof of the theorem}

Consider all expectations in this proof to be taken with respect to the choice of subset of coordinates $S^k$.
Using Lemma~\ref{l:mirror_acd} we have
\begin{eqnarray*}
&&
\eta \sum_{i\in S^k} \langle \frac{1}{p_i}\nabla_i f(x^{k+1}) e_i , z^{k}-u  \rangle -\frac{\eta\sigma_w}{2}\|x^{k+1}-u \|^2_w 
\\
&
\leq &
\eta \sum_{i\in S^k} \langle \frac{1}{p_i}\nabla_i f(x^{k+1}) e_i, z^{k}-z^{k+1}  \rangle
 -\frac12\| z^k-z^{k+1}\|^2_w+\frac{1}{2}\|z^k-u \|_w^2
 -\frac{1+\eta \sigma_w}{2}\| z^{k+1}-u\|^2_{w} 
\\
&\leq &
\frac{\eta^2}{2}\| \sum_{i\in S^k} \frac{1}{p_i}   \nabla_i f(x^{k+1})e_i \|^2_{w^{-1}}+\frac{1}{2}\|z^k-u \|_w^2
-\frac{1+\eta \sigma_w}{2}\| z^{k+1}-u\|^2_{w} 
\\
&= &
\frac{\eta^2}{2} \| \sum_{i\in S^k}  \nabla_i f(x^{k+1})e_i \|^2_{w^{-1}\circ p^{-2}}+\frac{1}{2}\|z^k-u \|_w^2
-\frac{1+\eta \sigma_w}{2}\| z^{k+1}-u\|^2_{w}  .
\end{eqnarray*}

Taking the expectation over the choice of $S^k$ we get

\begin{eqnarray}\nonumber
&&
\eta \langle \nabla f(x^{k+1}), z^{k}-u  \rangle -\frac{\eta\sigma_w}{2}\|x^{k+1}-u \|^2_w 
\\ \nonumber
& \leq &
\frac{\eta^2}{2} \|  \nabla f(x^{k+1}) \|^2_{w^{-1}\circ p^{-1}}+\frac{1}{2}\|z^k-u \|_w^2
-\frac{1+\eta \sigma_w}{2}\E{\| z^{k+1}-u\|^2_{w} }
\\ \nonumber
& 
\stackrel{\eqref{eq:w_def}}{=} &
\frac{\eta^2}{2} \|  \nabla f(x^{k+1}) \|^2_{v^{-1}\circ p}+\frac{1}{2}\|z^k-u \|_w^2
-\frac{1+\eta \sigma_w}{2}\E{\| z^{k+1}-u\|^2_{w} }
\\
\nonumber & 
\stackrel{\eqref{eq:eso_inq}}{\leq} &
\eta^2\left( f(x^{k+1})-\E{f(y^{k+1})}\right)+\frac{1}{2}\|z^k-u \|_w^2
-\frac{1+\eta \sigma_w}{2}\E{\| z^{k+1}-u\|^2_{w} }. \label{eq:acd_proof_almost}
\end{eqnarray}

Next we can do the following bounds
\begin{eqnarray*}
& & \eta \left( f(x^{k+1})- f(x^*)\right)
 \\
& & \qquad \qquad 
\stackrel{\eqref{eq:sc}}{\leq}
\eta \langle \nabla f(x^{k+1}),x^{k+1}-x^*\rangle -\frac{\eta \sigma_w}{2} \|x^*-x^{k+1} \|^2_w
\\
& & \qquad \qquad 
=
\eta \langle \nabla f(x^{k+1}),x^{k+1}-z^{k}\rangle +\eta \langle \nabla f(x^{k+1}),z^{k}-x^*\rangle -\frac{\eta \sigma_w}{2} \|x^*-x^{k+1} \|^2_w
\\
& & \qquad \qquad 
\stackrel{\eqref{eq:x_update_acd}}{=}
\frac{(1-\theta)\eta}{\theta} \langle \nabla f(x^{k+1}),y^k-x^{k+1}\rangle +\eta \langle \nabla f(x^{k+1}),z^{k}-x^*\rangle -\frac{\eta \sigma_w}{2} \|x^*-x^{k+1} \|^2_w
\\
& &\qquad \qquad  
\stackrel{\eqref{eq:acd_proof_almost}}{\leq}
\frac{(1-\theta)\eta}{\theta}\left( f(y^k)-f(x^{k+1})\right)+\eta^2\left(f(x^{k+1}) -\E{f(y^{k+1})}\right)
\\
&& \qquad\qquad \qquad+
\frac{1}{2}\|z^k-x^* \|_w^2
-\frac{1+\eta \sigma_w}{2}\E{\| z^{k+1}-x^*\|^2_{w} }.
\end{eqnarray*}
Choosing $\eta=\frac1\theta$  and rearranging the above we obtain

 \[
\frac{1}{\theta^2}\left(\E{f(y^{k+1})}-f(x^*)\right)+\frac{1+\frac{\sigma_w}{\theta}}{2}\E{\| z^{k+1}-x^*\|^2_{w} } \leq
\frac{(1-\theta)}{\theta^2}\left( f(y^k)-f(x^*)\right)+
\frac{1}{2}\|z^k-x^* \|_w^2.
 \]
Finally, setting $\theta$ such that $1+\frac{\sigma_w}{\theta}=\frac{1}{1-\theta}$, which coincides with~\eqref{eq:tau_def_acd}, we get
\[
\E{P^{k+1}}
\leq
(1-\theta)P^k,
\]
as desired.

\section{Better rates for mini-batch \texttt{CD} (without acceleration) } \label{sec:cd_imp}

In this section we establish better rates for for mini-batch \texttt{CD} method than the current state of the art. Our starting point is the following complexity theorem.

\begin{theorem} Choose any proper sampling and let $\mP$ be its probability matrix and $p$ its probability vector. Let \[c(S,\mM) \eqdef \lambda_{\max}(\mP''\circ \mM),\] where $\mP'' \eqdef \mD^{-1} \mP \mD^{-1}$ and $\mD \eqdef \Diag{p}$. Then the vector $v$ defined by $v_i = c(S,\mM) p_i$ satisfies  the ESO inequality \eqref{eq:v_def}. Moreover, if we run the non-accelerated \texttt{CD} method \eqref{eq:Parallel-CD-intro}  with this sampling and stepsizes $\alpha_i = \tfrac{1}{c(S,\mM) p_i}$, then the iteration complexity of the method is
\begin{equation}\label{eq:bis798f98gud}\frac{c(S,\mM)}{\sigma} \log \frac{1}{\epsilon}.\end{equation}
\end{theorem}
\begin{proof} Let $v_i = c p_i$ for all $i$. The ESO inequality holds for this choice of $v$ if $\mP \circ \mM \preceq c \mD^2$. This is equivalent to
Since $\mD^{-1}(\mP \circ \mM)\mD^{-1} = \mP'' \circ \mM$, the above inequality is equivalent to $\mP'' \circ \mM \preceq c \mI$, which is equivalent to $c\geq \lambda_{\max}(\mP'' \circ \mM)$. So, choosing $c=c(S,\mM)$ works.  Plugging this choice of $v$ into the complexity result \eqref{eq:NSync} gives \eqref{eq:bis798f98gud}.
\end{proof}

\subsection{Two uniform samplings and one new importance sampling}

In the next theorem we compute now consider several special samplings. All of them choose in expectation a mini-batch of size $\tau$ and are hence directly comparable.

\begin{theorem} \label{thm:809h0s9s} The following statements hold:
\begin{itemize}
\item[(i)] Let $S_1$ be the $\tau$--nice sampling. Then 
\begin{equation}\label{eq:c1}c_1\eqdef c(S_1,\mM) = \frac{n}{\tau} \lambda_{\max}\left( \frac{\tau-1}{n-1}\mM + \frac{n-\tau}{n-1} \Diag{\mM}\right).\end{equation}
\item[(ii)] Let $S_2$ be the independent uniform sampling with mini-batch size $\tau$. That is, for all  $i$ we independently decide whether $i\in S$, and do so by picking $i$ with probability $p_i = \tfrac{\tau}{n}$. Then
\begin{equation}\label{eq:c2}c_2\eqdef c(S_2,\mM) = \lambda_{\max}\left(\mM + \frac{n-\tau}{\tau}\Diag{\mM}\right).\end{equation}
\item[(iii)] Let $S_3$ be an independent sampling where we choose $p_i \propto \tfrac{\mM_{ii}}{\delta + \mM_{ii}}$ where $\delta>0$ is chosen so that $\sum_i p_i = \tau$.  Then
\begin{equation}\label{eq:c3} c_3\eqdef c(S_3,\mM) = \lambda_{\max}(\mM) + \delta.\end{equation}
Moreover,
\begin{equation}\label{eq:bi7gdv98dg} \delta \leq \frac{\trace{\mM}}{\tau}.\end{equation}
\end{itemize}
\end{theorem}
\begin{proof}
We will deal with each case separately:
\begin{enumerate}
\item[(i)] The probability matrix of $S_1$ is $\mP = \frac{\tau}{n}\left(\beta \mE + (1-\beta) \mI  \right),$ where $\beta = \tfrac{\tau-1}{n-1}$, and $\mD = \tfrac{\tau}{n}\mI$. Hence,
\begin{eqnarray*}\mP'' \circ \mM &=& (\mD^{-1} \mP \mD^{-1})\circ \mM  \\
&=& \frac{\tau}{n}\left( \beta \mD^{-1}\mE \mD^{-1} + (1-\beta) \mD^{-2} \right) \circ \mM\\
&=& \frac{\tau}{n}\left( \frac{\tau-1}{n-1}\mE + \frac{n-\tau}{n-1} \mI \right) \circ \mM \\
&=& \frac{\tau}{n}\left( \frac{\tau-1}{n-1}\mM + \frac{n-\tau}{n-1} \Diag{\mM}\right).
\end{eqnarray*}
\item[(ii)] The probability matrix of $S_2$ is $\mP = \frac{\tau}{n}\left(\frac{\tau}{n}\mE + (1-\frac{\tau}{n}) \mI  \right)$, and $\mD = \tfrac{\tau}{n}\mI$. Hence,
\begin{eqnarray*}\mP'' \circ \mM &=& (\mD^{-1} \mP \mD^{-1})\circ \mM  \\
&=&  \frac{\tau}{n}\left(\frac{\tau}{n}\mD^{-1}\mE \mD^{-1}+ \left(1-\frac{\tau}{n}\right) \mD^{-2}  \right) \circ \mM\\
&=& \left( \mE + \frac{n-\tau}{\tau} \mI \right) \circ \mM \\
&=& \mM + \frac{n-\tau}{\tau} \Diag{\mM}.
\end{eqnarray*}
\item[(iii)] The probability matrix of $S_3$ is $\mP = pp^\top + \mD - \mD^2$. Therefore,
\begin{eqnarray*}\mP'' \circ \mM &=& (\mD^{-1} \mP \mD^{-1})\circ \mM  \\
&=&   \left(\mD^{-1} pp^\top \mD^{-1}+ \mD^{-1} - \mI \right) \circ \mM\\
&=&   \left(\mE + \mD^{-1} - \mI \right) \circ \mM\\
&=&   \left(\mE + \delta(\Diag{\mM})^{-1} \right) \circ \mM\\
&=& \mM+ \delta \mI.
\end{eqnarray*}
To establish the bound on $\delta$, it suffices to note that
\[\tau = \sum_i p_i = \sum_{i} \frac{\mM_{ii}}{\delta + \mM_{ii}} \leq \sum_{i} \frac{\mM_{ii}}{\delta } = \frac{\trace{\mM}}{\delta}.\]
\end{enumerate}
\end{proof}

\subsection{Comparing the samplings}

In the next result we show that sampling $S_3$ is at most twice worse than $S_2$, which is at most twice worse than $S_1$. Note that $S_1$ is uniform; and it is the standard mini-batch sampling used in the literature and applications. Our novel sampling $S_3$ is {\em non-uniform}, and is at most four times worse than $S_1$ in the worst case. However, {\em it can be substantially better}, as we shall show later by giving an example.

\begin{theorem} \label{thm:nonacc_comp}The leading complexity terms $c_1,c_2$, and $c_3$ of \texttt{CD} (Algorithm~\eqref{eq:Parallel-CD-intro}) with samplings $S_1, S_2$, and $S_3$, respectively, defined in Theorem~\ref{thm:809h0s9s}, compare as follows:
\begin{itemize}
	\item[(i)] $c_3 \leq \frac{2n-\tau}{n-\tau} c_2$
	\item[(ii)] $c_2 \leq \frac{(n-1)\tau}{n(\tau-1)} c_1 \leq 2 c_1$
\end{itemize}
\end{theorem}
\begin{proof}
We have:
\begin{enumerate}
\item[(i)] 
\begin{eqnarray*}
c_3 &\overset{\eqref{eq:c3}}{=}&\lambda_{\max}(\mM) + \delta \\
&\leq & \lambda_{\max}\left( \mM + \frac{n-\tau}{\tau}\Diag{\mM}\right) + \delta \\
&\overset{\eqref{eq:c2}}{=}& c_2 + \delta \\
&\overset{\eqref{eq:bi7gdv98dg}}{\leq }&  c_2 + \frac{\trace{\mM}}{\tau}\\
&\leq & c_2 + \frac{n \max_i \mM_{ii}}{\tau} \\
&= & c_2 + \frac{n}{n-\tau}\frac{n-\tau}{\tau} \max_i \mM_{ii} \\
&= & c_2 + \frac{n}{n-\tau}\lambda_{\max}\left(\frac{n-\tau}{\tau} \Diag{\mM}\right)\\
&\leq & c_2 + \frac{n}{n-\tau}\lambda_{\max}\left(\mM+\frac{n-\tau}{\tau} \Diag{\mM}\right)\\
&\overset{\eqref{eq:c2}}{=}&  \frac{2n-\tau}{n-\tau} c_2.
\end{eqnarray*}
\item[(ii)] 
\begin{eqnarray*} 
c_2 &\overset{\eqref{eq:c2}}{=}& \lambda_{\max}\left(\mM + \frac{n-\tau}{\tau}\Diag{\mM}\right) \\
&=&  \lambda_{\max}\left( \frac{n(\tau-1)}{\tau(n-1)} \mM + \frac{n-\tau}{\tau}\Diag{\mM} + \left(1-\frac{n(\tau-1)}{\tau(n-1)}\right)\mM\right)\\
&\overset{(\dagger)}{\leq} &  \lambda_{\max}\left( \frac{n(\tau-1)}{\tau(n-1)} \mM + \frac{n-\tau}{\tau}\Diag{\mM}\right) + \lambda_{\max}\left(\left(1-\frac{n(\tau-1)}{\tau(n-1)}\right)\mM\right)\\
&\leq & \lambda_{\max}\left( \frac{n(\tau-1)}{\tau(n-1)} \mM + \frac{n(n-\tau)}{\tau(n-1)}\Diag{\mM}\right) +\frac{n-\tau}{(n-1)\tau} \lambda_{\max}\left(\mM\right)\\
&\overset{\eqref{eq:c1}}{=}& c_1 + \frac{n-\tau}{(n-1)\tau} \lambda_{\max}\left(\mM\right)\\
&\overset{\eqref{eq:c2}}{\leq} & c_1 +  \frac{n-\tau}{(n-1)\tau} c_2.
\end{eqnarray*}
The statement follows by reshuffling the final inequality. In  step $(\dagger)$ we have used subadditivity of the function $\mA \mapsto \lambda_{\max}(\mA)$.
\end{enumerate}
\end{proof}

The next simple example shows that sampling $S_3$ can be arbitrarily better than sampling $S_1$.

\begin{example}\label{ex:nonacc_imp_diff}
Consider $n>>1$, and choose any $\tau$ and 
\[
\mM\eqdef\begin{pmatrix}
n&0^\top \\ 0 & \mI 
\end{pmatrix}
\]
for $\mI\in \R^{(n-1)\times (n-1)}$. Then, it is easy to verify that $c_1\stackrel{\eqref{eq:c1}}{=}\frac{n^2}{\tau}$ and $c_3\stackrel{\eqref{eq:c3}+\eqref{eq:bi7gdv98dg}}{\leq} n+\frac{2n-1}{\tau}=\cO(\frac{n}{\tau})$. Thus, convergence rate of \texttt{CD} with $S_3$ sampling can be up to $\cO(n)$ times better than convergence rate of \texttt{CD} with $\tau$--nice sampling. 
\end{example}

\begin{remark}\label{rem:proport}
Looking only at diagonal elments of $\mM$, an intuition tells us that one should sample a coordinate corresponding to larger diagonal entry of $\mM$ with higher probability. However, this might lead to worse convergence, comparing to $\tau$--nice sampling. Therefore the results we provide in this section cannot be qualitatively better, i.e. there are examples of smoothness matrix, for which assigning bigger probability to bigger diagonal elements leads to worse rate. It is an easy exercise to verify that for $\mM\in \R^{10\times 10}$ such that
\[
\mM\eqdef \begin{pmatrix}
2 & 0^\top \\ 0& 11^\top
\end{pmatrix}, 
\] 
and $\tau\geq 2$ we have $c(S_{\text{nice}},\mM)\leq c(S',\mM)$ for any $S'$ satisfying $p(S')_i\geq p(S')_j$ if and only if $\mM_{ii}\geq \mM_{jj}$. 
\end{remark}

\section{Proofs for Section~\ref{sec:import}}
\subsection{Proof of Theorem~\ref{thm:LB}} \label{app:inequality_lower_bound}

We start with a lemma which allows us to focus on ESO parameters $v_i$ which are proportional to the squares of the probabilities $p_i$.

\begin{lemma} \label{lem:b7f8gf8} Assume that the ESO inequality  \eqref{eq:v_def}  holds. Let $j=\arg\max_i \frac{v_i}{p_i^2}$, $c = \frac{v_j}{p_j^2}$ and $v' = c p^2$ (i.e.,  $v'_i = c p_i^2$ for all $i$). Then the following statements hold:
\begin{enumerate}
\item[(i)] $v'\geq v$.
\item[(ii)] ESO inequality \eqref{eq:v_def}  holds for $v'$ also.
\item[(iii)]  Assuming $f$ is $\sigma$--convex, Theorem~\ref{th:acd} holds if we replace $v$ by $v'$, and the rate \eqref{eq:ineq-minibatch-speedup} is unchanged if we replace $v$ by $v'$. 
\end{enumerate}
\end{lemma}

\begin{proof}
\begin{enumerate}
\item[(i)] $v'_i = c p_i^2 = \frac{v_j}{p_j^2} p_i^2 = \left( \frac{v_j}{p_j^2} \frac{p_i^2}{v_i}\right) v_i \geq v_i$.
\item[(ii)] This follows directly from (i).
\item[(iii)]  Theorem~\ref{th:acd} holds with $v$ replaced by $v'$ because ESO holds. To show that the rates are unchanged first note that
$\max_i \frac{v_i}{p_i^2} = \frac{v_j}{p_j^2} = c$. On the other hand, by construction, we have $c= \frac{v'_i}{p_i^2}$ for all $i$. So, in particular, $c = \max_i \frac{v'_i}{p_i^2}$.
\end{enumerate}
\end{proof}

In view of the above lemma, we can assume without loss of generality that $v=cp^2$. Hence, the rate in \eqref{eq:ineq-minibatch-speedup}  can be written in the form
\begin{equation}\label{eq:b87gs98scc}\sqrt{\max_i \frac{v_i}{p_i^2 \sigma}}  = \sqrt{\frac{c}{\sigma}}.\end{equation}

In what follows, we will establish a lower  bound on $c$, which will  lead to the lower bound on the rate expressed as inequality~\eqref{eq:ineq-minibatch-speedup}. As a starting point, note that directly from \eqref{eq:v_def} we get the bound
\begin{equation}\label{eq:nb87d90sh}\mP \circ \mM \preceq  \Diag{p\circ v} = c \Diag{p^3}.\end{equation}
Let $\mD_1=\Diag{p}^{-1/2}$ and $\mD_2 = \Diag{p}^{-1}$. From \eqref{eq:nb87d90sh} we get
$ \mD_1 \mD_2 (\mP\circ \mM) \mD_2 \mD_1 \preceq  c \mI$ and hence
\begin{equation} \label{eq:bu987g98f} c \geq c(S,\mM) \eqdef \lambda_{\max}(\mD_1 \mD_2 (\mP\circ \mM) \mD_2 \mD_1).\end{equation}

At this point, the following identity will be useful.

\begin{lemma}\label{lem:bgfd78gd8} Let $\mA,\mB,\mD_1,\mD_2\in \R^{n\times n}$, with $\mD_1,\mD_2$ being diagonal. Then
\begin{equation} \label{eq:hadamard_diag}
\mD_1(\mA \circ \mB) \mD_2 = (\mD_1\mA \mD_2) \circ \mB =   \mA \circ(\mD_1\mB \mD_2).
\end{equation}
\end{lemma}
\begin{proof}
The proof is straightforward, and hence we do not include it. The identity is formulated as an exercise in \cite{FuzhenZhang}.
\end{proof}

Repeatedly applying Lemma~\ref{lem:bgfd78gd8}, we get
\[\mD_1 \mD_2 (\mP\circ \mM) \mD_2 \mD_1 = \underbrace{(\mD_1 \mP \mD_1)}_{\mP'} \circ \underbrace{(\mD_2 \mM \mD_2)}_{\mM'}.\]
Plugging this back into \eqref{eq:bu987g98f},  and since $\mP'_{ii}=1$ for all $i$, we get the bound
\begin{eqnarray} \label{eq: bi9s8g9snnn} c &\geq& c(S,\mM) =  \lambda_{\max}(\mP' \circ \mM') \geq  \max_i \; (\mP'\circ \mM')_{ii} =  \max_i \mP'_{ii} \mM'_{ii}  = \max_i \mM'_{ii} \notag \\
&=& \max_i \frac{\mM_{ii}}{p_i^2}   \geq   \frac{\left(\sum_{i=1}^n \mM_{ii}^{1/2}\right)^2}{\tau^2}.\label{eq:bdg8db8d}\end{eqnarray}
The last inequality follows by observing that the optimal solution of the optimization problem
\[\min_p \left\{\max_i \frac{\mM_{ii}}{p_i^2} \;|\; p_1,\dots,p_n > 0, \; \sum_i p_i =\tau\right\}\]
is $p_i =\tau \frac{\mM_{ii}^{1/2}}{\sum_j \mM_{jj}^{1/2}} $. Inequality~\eqref{eq:ineq-minibatch-speedup} now follows by substituting  
the lower bound on $c$ obtained in \eqref{eq:bdg8db8d} into 
\eqref{eq:b87gs98scc}.

\subsection{Proof of Lemma~\ref{thm:special-ESO-result}}
\begin{eqnarray*}
 \Diag{p_1 v_1, \dots, p_n v_n} &=& c(S,\mM)\Diag{p_1^3,\dots ,p_n^3}\\
 &=&
c(S,\mM) \mD^{3} \\
 &=&
\lambda_{\max}\left(\left( \mD^{-1/2} \mP \mD^{-1/2}\right) \circ \left( \mD^{-1}\mM \mD^{-1}\right)\right)\mD^{3}
\\
& \succeq & 
\mD^{\frac32}\left(\left( \mD^{-1/2} \mP \mD^{-1/2}\right) \circ \left( \mD^{-1}\mM \mD^{-1}\right)\right)\mD^{\frac32}
\\
&\stackrel{\eqref{eq:hadamard_diag}}{=}&
\mP\circ \mM.
\end{eqnarray*}
The last inequality came from the fact that $\mD$ is diagonal. 

\subsection{Bound on $c(S_1,\mS)$}

\begin{lemma}\label{lem:tau-nice-2nd-derivation}
$c(S_1,\mM) \leq \tfrac{n^2}{\tau^2} ((1-\beta)\max_i \mM_{ii} + \beta L).$
\end{lemma}
\begin{proof} Recall that the probability matrix of $S_1$ is $\mP = \tfrac{\tau}{n}\left((1-\beta) \mI + \beta \mE\right)$. Since $p_i=\tfrac{\tau}{n}$ and $\mM \preceq L \mI$, we have
\begin{eqnarray*}
c(S_1,\mM) &=& \lambda_{\max}\left(\mP' \circ \mM' \right) \\
&=& \lambda_{\max}\left( (\mD^{-1/2}\mP \mD^{-1/2}) \circ (\mD^{-1}\mM \mD^{-1}) \right) \\ 
&=& \lambda_{\max} \left( \tfrac{\tau}{n}\left((1-\beta) \mD^{-1} + \beta \mD^{-1/2}\mE \mD^{-1/2}\right) \circ \mD^{-1} \mM \mD^{-1} \right)\\
&=&\tfrac{\tau}{n} \lambda_{\max} \left( \left((1-\beta) \mD^{-1} + \beta \mD^{-1/2}\mE \mD^{-1/2}\right) \circ \mD^{-1} \mM \mD^{-1} \right)\\
&=& \tfrac{\tau}{n}  \lambda_{\max}  \left((1-\beta) \Diag{\mM_{ii}/p_i^3} + \beta \mD^{-3/2}\mM \mD^{-3/2}\right) \\
&\preceq  & \tfrac{\tau}{n}  \lambda_{\max}  \left((1-\beta) \Diag{\mM_{ii}/p_i^3} + \beta L \mD^{-3} \right) \\
&=& \tfrac{\tau}{n}  \lambda_{\max}  \left((1-\beta) \tfrac{n^3}{\tau^3}\max_i \mM_{ii}+ \beta L \tfrac{n^3}{\tau^3} \right) \\
&=& \tfrac{n^2}{\tau^2} \left((1-\beta) \max_i \mM_{ii} + \beta L \right).
\end{eqnarray*}
\end{proof}

\subsection{Proof of Theorem~\ref{thm:comparison}}
For the purpose of this proof, let $S_2$ be the independent uniform sampling with mini-batch size $\tau$. That is, for all  $i$ we independently decide whether $i\in S$, and do so by picking $i$ with probability $p_i = \tfrac{\tau}{n}$. Recall that $S_3$ is the independent importance sampling.

For simplicity, let  $\mP_i$ be the probability matrix of sampling $S_i$,  $\mD_i\eqdef \Diag{\mP_i}$, and $\mM'_i\eqdef\mD^{-1/2}_i\mM \mD^{-1/2}_i$, for $i=1,3$. Next, we have
\begin{eqnarray}
c(S_i,\mM)&=& \lambda_{\max}\left(\left( \mD_i^{-1/2} \mP_i \mD_i^{-1/2}\right) \circ \left( \mD_i^{-1}\mM \mD_i^{-1}\right)\right)
\nonumber
\\
&\stackrel{\eqref{eq:hadamard_diag}}{=}& \lambda_{\max}\left(\left( \mD_i^{-1} \mP_i \mD_i^{-1}\right) \circ \left( \mD_i^{-1/2}\mM \mD_i^{-1/2}\right)\right)
\nonumber
\\
\nonumber
&=& \lambda_{\max}\left(\left( \mE + \mD_i^{-1}-\mI  \right) \circ \mM_i'\right)
\\
&=& \lambda_{\max}\left(\mM_i' + \Diag{\mM_i'}\circ(\mD_i^{-1}-\mI ) \right),\label{eq:cs_simple}
\end{eqnarray}
where the third identity holds since both $S_i$ is an independent sampling, which means that $\left( \mD_i^{-1} \mP_i \mD_i^{-1}\right)_{kl} = \tfrac{p_{kl}}{p_{k}p_l}$, where $p=\Diag{\mD_i}$.

Denote $c_i\eqdef c(S_i,\mM)$. Thus for $S_2$ we have
\begin{equation}
c_2 =\frac{\tau}{n}\lambda_{\max}\left( \mM +\frac{n-\tau}{\tau} \Diag{\mM} \right).\label{eq:c2_acc}
\end{equation}

Let us now establish a technical lemma.

\begin{lemma}\label{lem:import_main}
\begin{equation}\label{eq:p_choice_ineq}
 \lambda_{\max}\left(\mM'_3 +\diag(\mM'_3)\circ(\mD^{-1}_3-\mI ) \right)
\leq
 \frac{2n-\tau}{n-\tau}\lambda_{\max}\left(\mM'_3 +\frac{n-\tau}{\tau} \Diag{\mM'_3}\right)
\end{equation}
\end{lemma}
\begin{proof}
The statement follows immediately repeating the steps of the proof of (i) from Theorem~\ref{thm:nonacc_comp} using the fact that for sampling $S_3$ we have $p_i/\mM_{ii}\propto p_i^{-1}-1$. 
\end{proof}

We can now proceed with comparing $c_2$ to $c_3$.
\begin{eqnarray}
c_3 &=&  \lambda_{\max}\left(\mM'_3 + \Diag{\mM'_3}\circ(\mD^{-1}_3-\mI ) \right) \nonumber \\
&\stackrel{\eqref{eq:p_choice_ineq}}{\leq}  &
 \frac{2n-\tau}{n-\tau}\lambda_{\max}\left(\mM'_3 +\frac{n-\tau}{\tau} \Diag{\mM'_3}\right) \nonumber \\
 &\stackrel{(*)}{\leq}&
\frac{2n-\tau}{n-\tau}
\lambda_{\max}\left(n  \Diag{\mM'_3} +\frac{n-\tau}{\tau} \Diag{\mM'_3}\right) \nonumber\\
&=&
\frac{2n-\tau}{n-\tau} \frac{n\tau+n-\tau}{\tau}\lambda_{\max}\left( \Diag{\mM'_3}\right) \nonumber\\
&\stackrel{(**)}{\leq}&
\frac{2n-\tau}{n-\tau} \frac{n\tau+n-\tau}{\tau} \frac{\tau}{n}\lambda_{\max}\left( \Diag{\mM}\right) \nonumber\\
&\leq&
\frac{2n-\tau}{n-\tau} \frac{n\tau+n-\tau}{\tau} \frac{\tau}{n}\frac{\tau}{n-\tau}\lambda_{\max}\left(\mM+\frac{n-\tau}{\tau} \Diag{\mM}\right)  \nonumber \\
&\stackrel{\eqref{eq:cs_simple}}{=}&
\frac{2n-\tau}{n-\tau} \frac{n\tau+n-\tau}{\tau} \frac{\tau}{n-\tau} c_2 \label{eq:c3c2cmp}
\end{eqnarray}

Above, inequality  $(*)$ holds since for any $n\times n$ matrix $\mQ\succ 0$ we have $\mQ\preceq n\Diag{\mQ}$ and inequality $(**)$ holds since $(\mD_3)_{ii}\geq(\mD_3)_{jj} $ if and only if $\mM_{ii}\geq \mM_{jj}$ due to choice of $p$.

Let us now compare to $c_2$ and $c_1$. We have
\begin{eqnarray}
c_1&=& \lambda_{\max}\left(\left( \mD_1^{-1/2} \mP_1 \mD_1^{-1/2}\right) \circ \left( \mD_1^{-1}\mM \mD_1^{-1}\right)\right)
\nonumber
\\
&\stackrel{\eqref{eq:hadamard_diag}}{=}& \lambda_{\max}\left(\left( \mD_1^{-1} \mP_1 \mD_1^{-1}\right) \circ \left( \mD_1^{-1/2}\mM \mD_1^{-1/2}\right)\right)
\nonumber
\\
\nonumber
&=& \lambda_{\max}\left(\left( \frac{\tau-1}{n-1}\frac{n}{\tau} \mE + \frac{n}{\tau} \mI-\frac{\tau-1}{n-1}\frac{n}{\tau}\mI  \right) \circ \mM''_1\right)
\\
\nonumber
&=& \frac{n}{\tau}\lambda_{\max}\left(\frac{\tau-1}{n-1}  \mM''_1 + \frac{n-\tau}{n-1} \Diag{ \mM''_1}  \right)
\\ \label{eq:c1_acc}
&=& \left(\frac{n}{\tau}\right)^2\lambda_{\max}\left(\frac{\tau-1}{n-1}  \mM + \frac{n-\tau}{n-1} \Diag{\mM}  \right).
\end{eqnarray}
As~\eqref{eq:c2_acc} and~\eqref{eq:c1_acc} are established, following the proof of (ii) from Theorem~\ref{thm:nonacc_comp}, we arrive at
\begin{equation}\label{eq:c1c2cmp}
c_2 \leq \frac{(n-1)\tau}{n(\tau-1)} c_1 \leq 2 c_1.
\end{equation}
It remains to combine~\eqref{eq:c3c2cmp} and~\eqref{eq:c1c2cmp} to establish~\eqref{eq:thm_comparison}. 

As an example where $c_3\approx \left(\frac{\tau}{n}\right)^2 c_2$, we propose Example~\ref{ex:acc_imp_diff}.

\begin{example}\label{ex:acc_imp_diff}
Consider $n\geq 1$, choose any $n\geq \tau \geq 1 $ and 
\[
\mM\eqdef\begin{pmatrix}
N &0^\top \\ 0 & \mI 
\end{pmatrix}
\]
for $\mI\in \R^{(n-1)\times (n-1)}$. Then, it is easy to verify that $c_1\stackrel{\eqref{eq:c1_acc}}{=}\left(\frac{n}{\tau}\right)^2N$. 
Moreover, for large enough $N$ we have 
\[
p\approx \left(1,\tfrac{\tau-1}{n-1},\dots, \tfrac{\tau-1}{n-1}\right)^\top \qquad \Rightarrow \qquad
\mM'_3 \approx  \Diag{ N, \frac{n-1}{\tau-1},\dots,\frac{n-1}{\tau-1} }. 
\]
Therefore, using~\eqref{eq:cs_simple} and again for large enough $N$, we get $c_3\approx N$. Thus, $c_3\approx \left(\frac{\tau}{n}\right)^2 c_2$. 
\end{example}

\clearpage
\section{Extra Experiments \label{sec:extra_exp}}

In this section we present additional numerical experiments. We first present some synthetic examples in Section~\ref{sec:artif} in order to have better understanding of both acceleration and importance sampling, and to see how it performs on what type of data. We also study how minibatch size influences the convergence rate. 

Then, in Section~\ref{exp:logreg}, we work with  logistic regression problem on LibSVM~\cite{chang2011libsvm} data.  For small datasets, we choose the parameters of \texttt{ACD} as theory suggests and for large ones, we estimate them, as we describe in the main body of the paper. Lastly, we tackle dual of SVM problem with squared hinge loss, which we present in Section~\ref{sec:SVM}. 

In most of plots we compare of both accelerated and non-accelerated \texttt{CD} with all samplings $S_1,S_2,S_3$ introduced in Sections~\ref{sec:sam1}, \ref{sec:sam2} and \ref{sec:sam3} respectively. We refer to \texttt{ACD} with sampling $S_3$ as \texttt{AN} (\texttt{A}ccelerated \texttt{N}onuniform), \texttt{ACD} with sampling $S_1$ as \texttt{AU}, \texttt{ACD} with sampling $S_2$ as \texttt{AN2}, \texttt{CD} with sampling $S_3$ as \texttt{NN}, \texttt{CD} with sampling $S_1$as \texttt{NU} and \texttt{CD} with sampling $S_2$ as \texttt{NN2}. As for Sampling 2, it might happen that probabilities become larger than one if $\tau$ is large (see Section~\ref{sec:sam2}), we set those probabilities to 1 while keeping the rest as it is. 

We compare the mentioned methods for various choices of the expected minibatch sizes $\tau$ and several problems.

\subsection{Synthetic quadratics~\label{sec:artif}}
As we mentioned, the goal of this section is to provide a better understanding of both acceleration and importance sampling. For this purpose we consider as simple setting as possible -- minimizing quadratic
\begin{equation}\label{eq:quadratic}
f(x)=\tfrac12 x^\top \mM x-b^\top x,
\end{equation}
where $b\sim N(0,I)$ and $\mM$ is chosen as one of the 5 types, as the following table suggests. 

\begin{table}[H]
\centering \small
\begin{tabular}{ |c|c|  }
 \hline
 {\bf Problem type} & $\mM$\\
 \hline
 \hline
1 & $\mA^\top \mA + \mI$ for $\mA^{\frac{n}{2}\times n}$;  have independent entries from $N(0,1)$
\\  \hline
2 & 
$\mA^\top \mA + \mI$ for $\mA^{2n\times n}$;  have independent entries from $N(0,1)$
\\ \hline
3 & 
$\diag(1,2,\dots,n)$
\\ \hline
4 & 
$\mA + \mI$, $\mA_{n,n}=n$, $\mA_{1:(n-1),1:(n-1)}=1$, $\mA_{1:(n-1),n}= \mA_{n,1:(n-1)}=0$
\\ \hline
5 & 
$\mA^\top \mD \mA + \mI$ for $\mA^{\frac{n}{2}\times n}$;  have independent entries from $N(0,1)$, $\mD=\frac{1}{\sqrt{n}}\Diag{1,2,\dots,n}$ 
\\ 
 \hline
\end{tabular}
\caption{Problem types}\label{tab:types}
\end{table}
In the first example we perform (Figure~\ref{fig:artif_1}), we compare the performance of both accelerated and non-accelerated algorithm with both nonuniform and $\tau$ nice sampling on problems as per Table~\ref{tab:types}. In all experiments, we set $n=1 000$ and we plot a various choices of $\tau$. 

\clearpage
\subsubsection{Comparison of methods on synthetic data \label{sec:marginal}}
Figure~\ref{fig:artif_1} presents the numerical performance of \texttt{ACD} for various types of synthetic problems given by~\eqref{eq:quadratic} and Table~\ref{tab:types}. It suggests what our theory shows -- that accelerated algorithm is always faster than its non-accelerated counterpart, and on top of that, performance of $\tau$--nice sampling ($S_1$) can be negligibly faster than  importance sampling ($S_2,S_3$), but is usually significantly slower. A significance of the importance sampling is mainly demonstrated on problem type 4, which roughly coincides with Examples~\ref{ex:nonacc_imp_diff} and~\ref{ex:acc_imp_diff}. Figure~\ref{fig:artif_1} presents Sampling 2 only for the cases when the bound on $\tau$ form Section~\ref{sec:sam2} is satisfied. 
\begin{figure}[H]
\centering
\begin{minipage}{0.25\textwidth}
  \centering
\includegraphics[width =  \textwidth ]{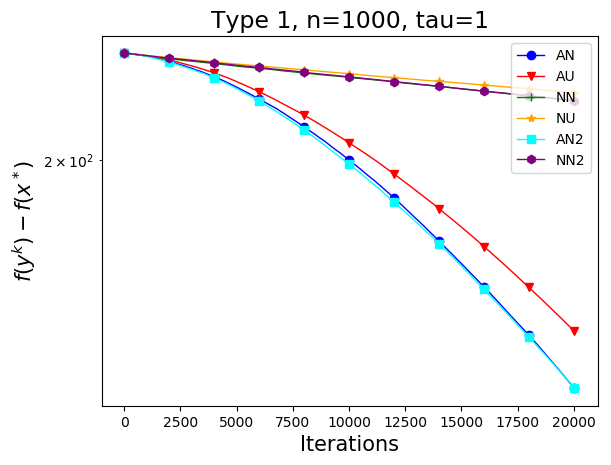}
\end{minipage}%
\begin{minipage}{0.25\textwidth}
  \centering
\includegraphics[width =  \textwidth ]{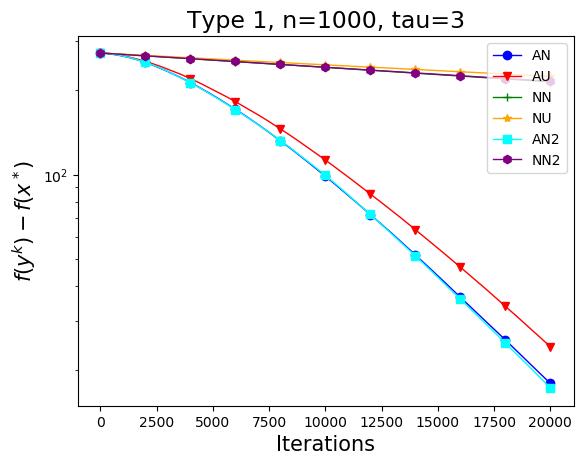}
\end{minipage}%
\begin{minipage}{0.25\textwidth}
  \centering
\includegraphics[width =  \textwidth ]{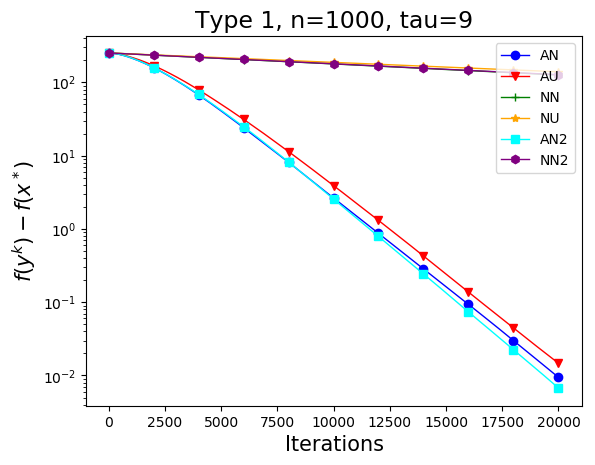}
\end{minipage}%
\begin{minipage}{0.25\textwidth}
  \centering
\includegraphics[width =  \textwidth ]{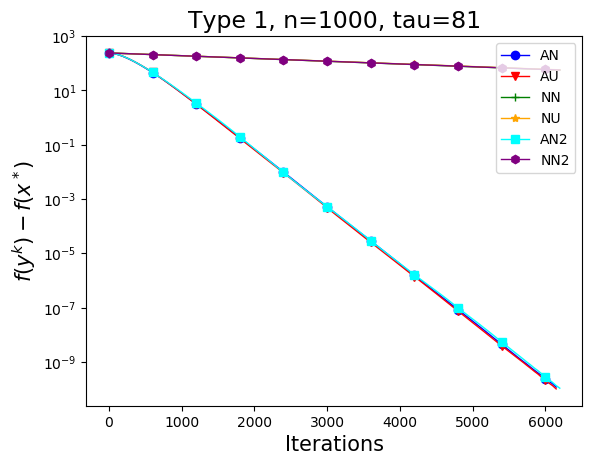}
\end{minipage}%
\\
\begin{minipage}{0.25\textwidth}
  \centering
\includegraphics[width =  \textwidth ]{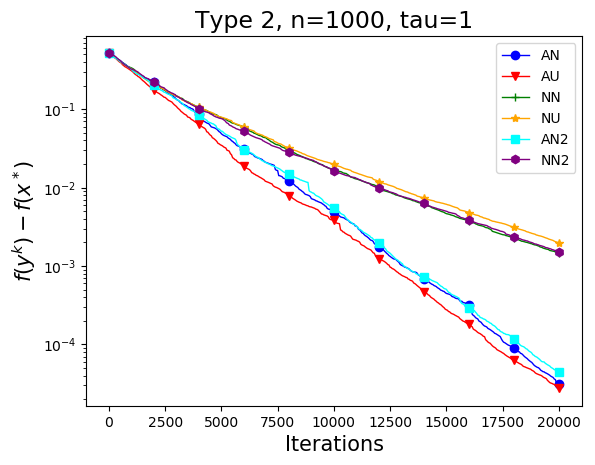}
\end{minipage}%
\begin{minipage}{0.25\textwidth}
  \centering
\includegraphics[width =  \textwidth ]{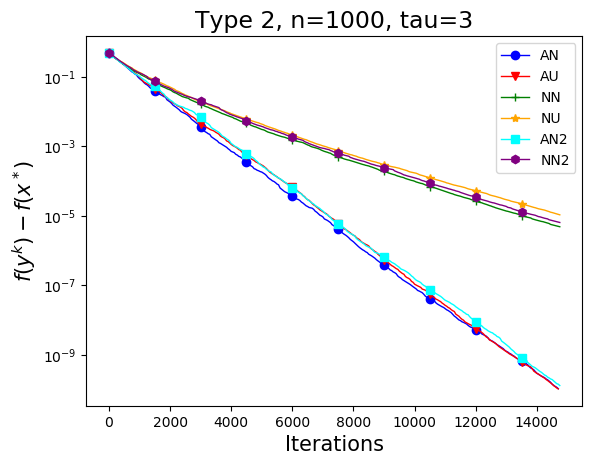}
\end{minipage}%
\begin{minipage}{0.25\textwidth}
  \centering
\includegraphics[width =  \textwidth ]{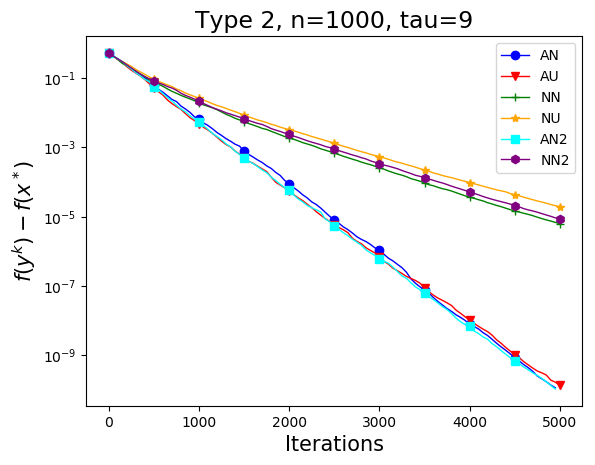}
\end{minipage}%
\begin{minipage}{0.25\textwidth}
  \centering
\includegraphics[width =  \textwidth ]{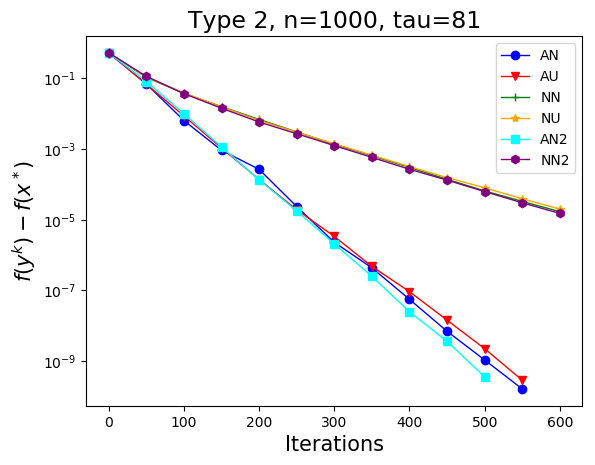}
\end{minipage}%
\\
\begin{minipage}{0.25\textwidth}
  \centering
\includegraphics[width =  \textwidth ]{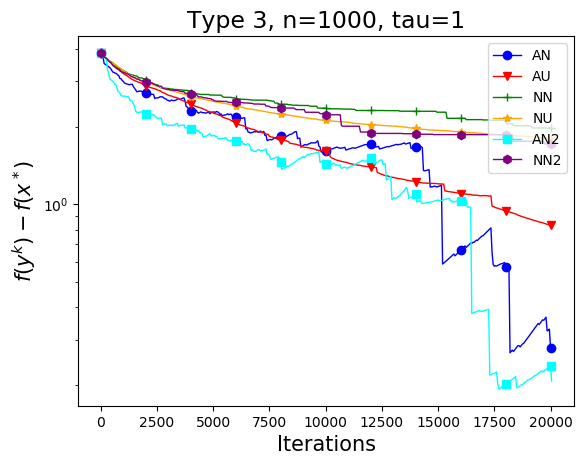}
\end{minipage}%
\begin{minipage}{0.25\textwidth}
  \centering
\includegraphics[width =  \textwidth ]{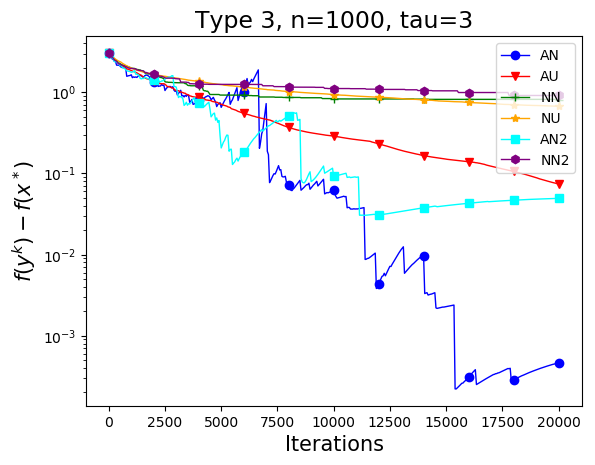}
\end{minipage}%
\begin{minipage}{0.25\textwidth}
  \centering
\includegraphics[width =  \textwidth ]{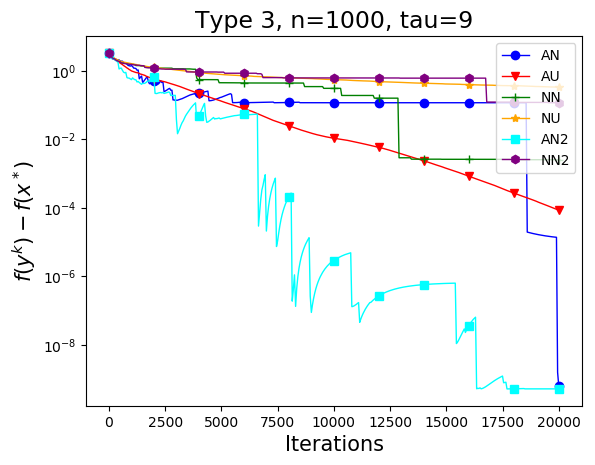}
\end{minipage}%
\begin{minipage}{0.25\textwidth}
  \centering
\includegraphics[width =  \textwidth ]{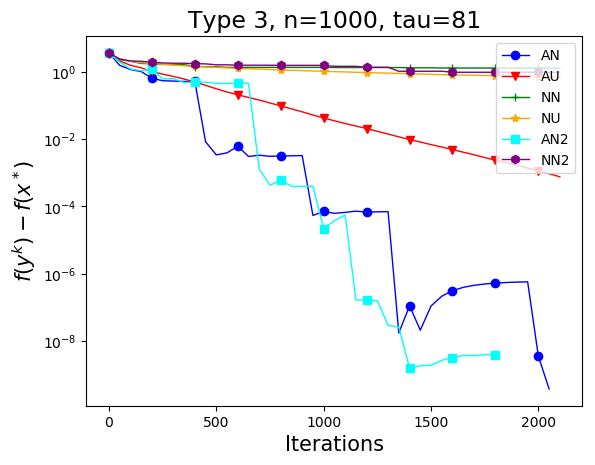}
\end{minipage}%
\\
\begin{minipage}{0.25\textwidth}
  \centering
\includegraphics[width =  \textwidth ]{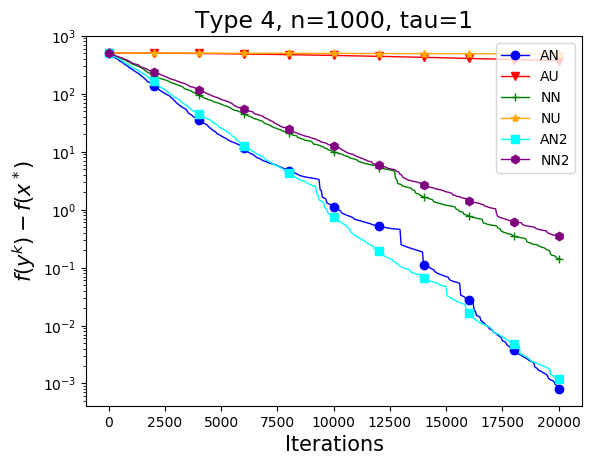}
\end{minipage}%
\begin{minipage}{0.25\textwidth}
  \centering
\includegraphics[width =  \textwidth ]{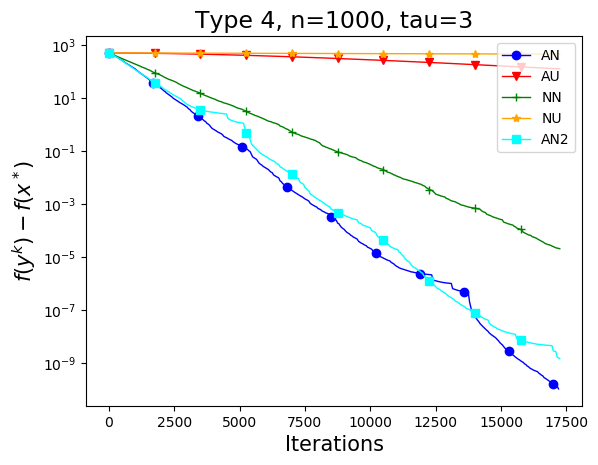}
\end{minipage}%
\begin{minipage}{0.25\textwidth}
  \centering
\includegraphics[width =  \textwidth ]{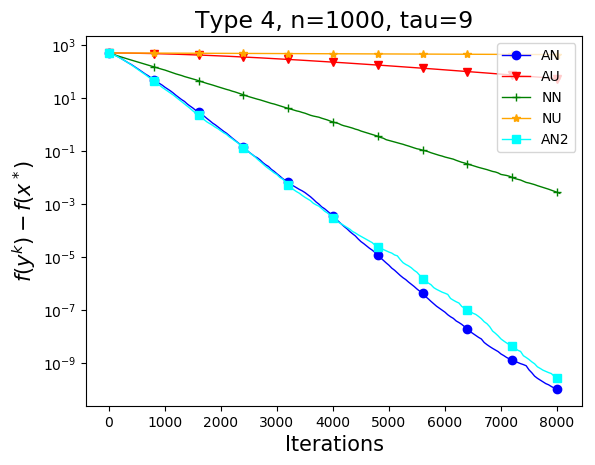}
\end{minipage}%
\begin{minipage}{0.25\textwidth}
  \centering
\includegraphics[width =  \textwidth ]{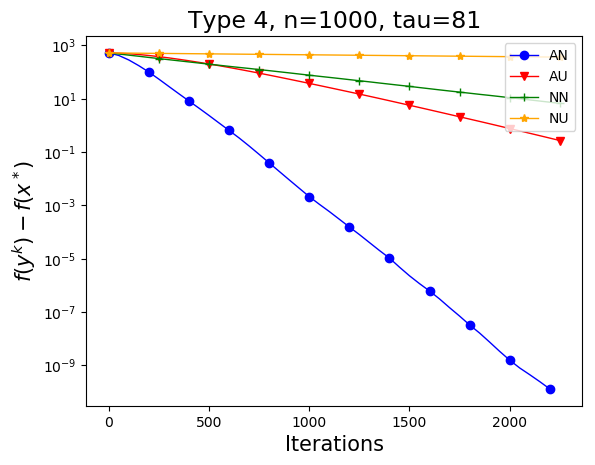}
\end{minipage}%
\\
\begin{minipage}{0.25\textwidth}
  \centering
\includegraphics[width =  \textwidth ]{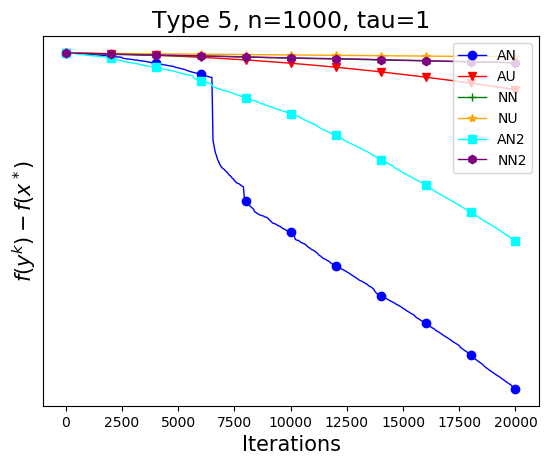}
\end{minipage}%
\begin{minipage}{0.25\textwidth}
  \centering
\includegraphics[width =  \textwidth ]{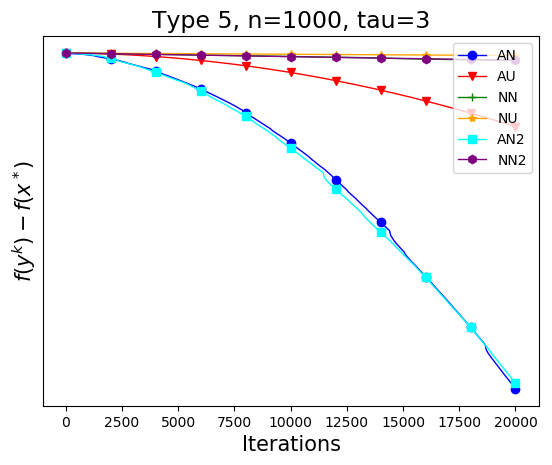}
\end{minipage}%
\begin{minipage}{0.25\textwidth}
  \centering
\includegraphics[width =  \textwidth ]{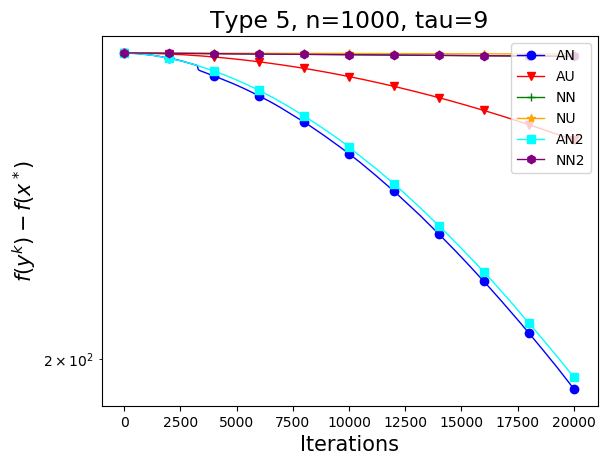}
\end{minipage}%
\begin{minipage}{0.25\textwidth}
  \centering
\includegraphics[width =  \textwidth ]{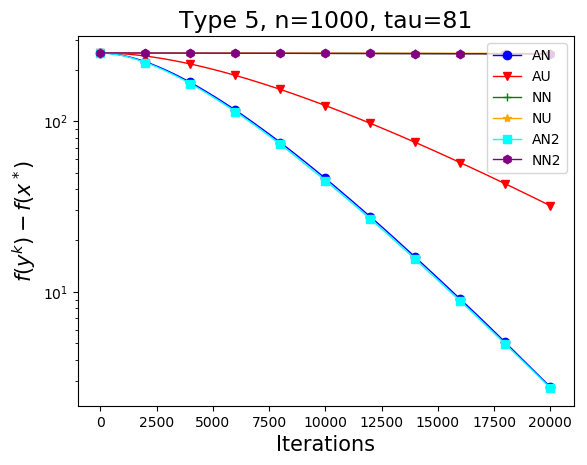}
\end{minipage}%
\\
\caption{Comparison of accelerated, nonaccelerated algorithm with both importance and $\tau$ nice sampling for a various quadratic problems.}\label{fig:artif_1}
\end{figure}

\clearpage

\subsubsection{Speedup in $\tau$ \label{sec:tau}}
The next experiment shows an empirical speedup for the coordinate descent algorithms for a various types of problems. For simplicity, we do not include Sampling 2. Figure~\ref{fig:tau_comp} provides the results. Oftentimes, the empirical speedup (in terms of the number of iteration) in $\tau$ is close to linear, which demonstrates the power and significance of minibatching.

\begin{figure}[H]
\centering
\begin{minipage}{0.25\textwidth}
  \centering
\includegraphics[width =  \textwidth ]{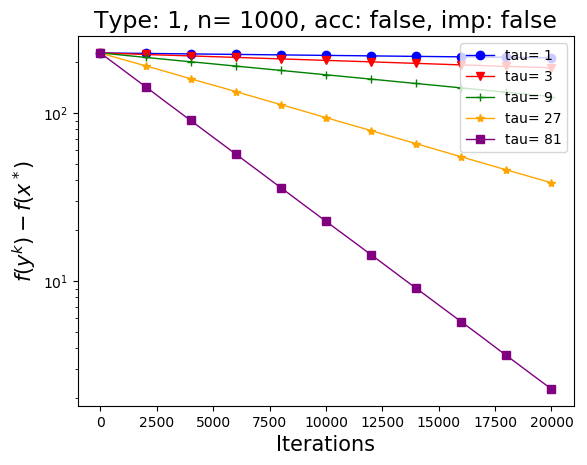}
\end{minipage}%
\begin{minipage}{0.25\textwidth}
  \centering
\includegraphics[width =  \textwidth ]{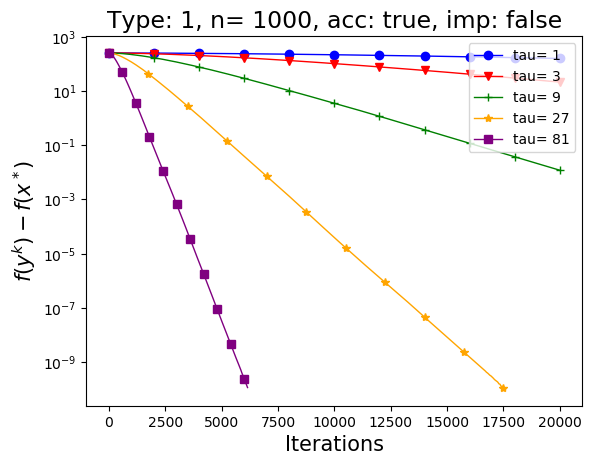}
\end{minipage}%
\begin{minipage}{0.25\textwidth}
  \centering
\includegraphics[width =  \textwidth ]{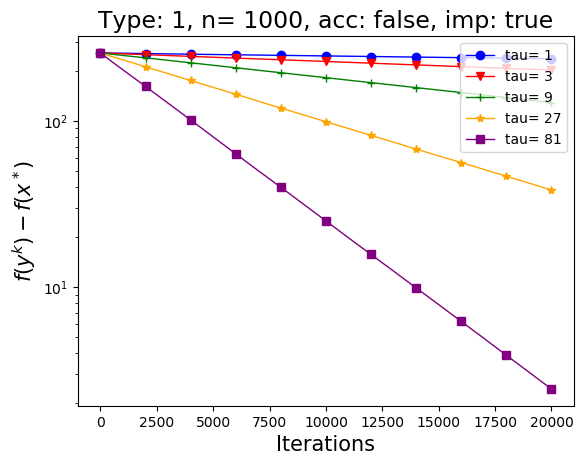}
\end{minipage}%
\begin{minipage}{0.25\textwidth}
  \centering
\includegraphics[width =  \textwidth ]{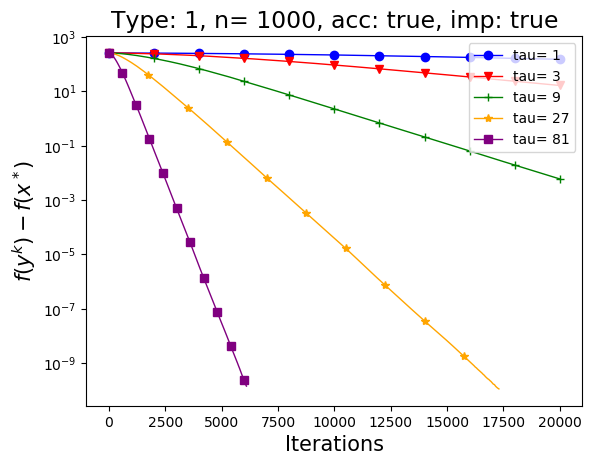}
\end{minipage}%
\\
\begin{minipage}{0.25\textwidth}
  \centering
\includegraphics[width =  \textwidth ]{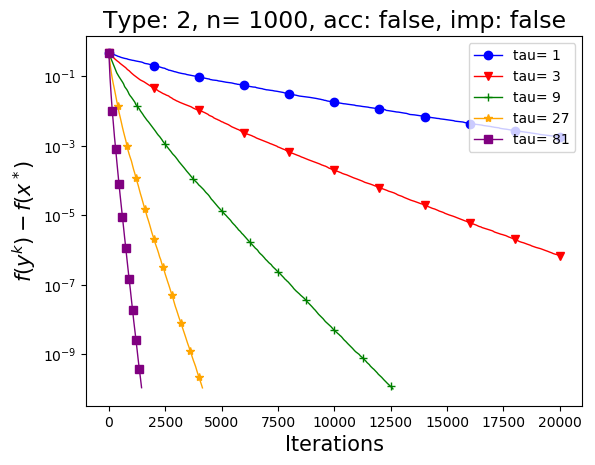}
\end{minipage}%
\begin{minipage}{0.25\textwidth}
  \centering
\includegraphics[width =  \textwidth ]{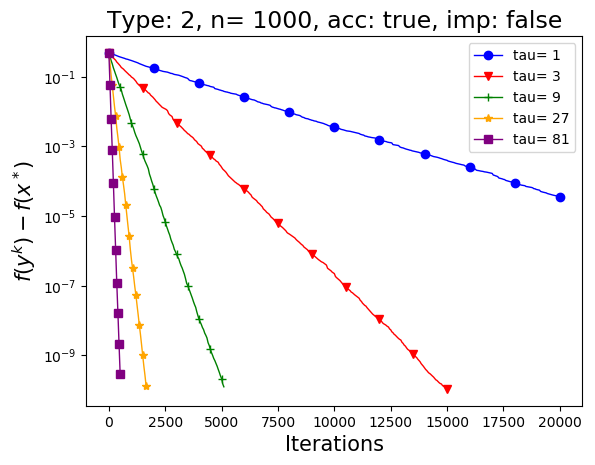}
\end{minipage}%
\begin{minipage}{0.25\textwidth}
  \centering
\includegraphics[width =  \textwidth ]{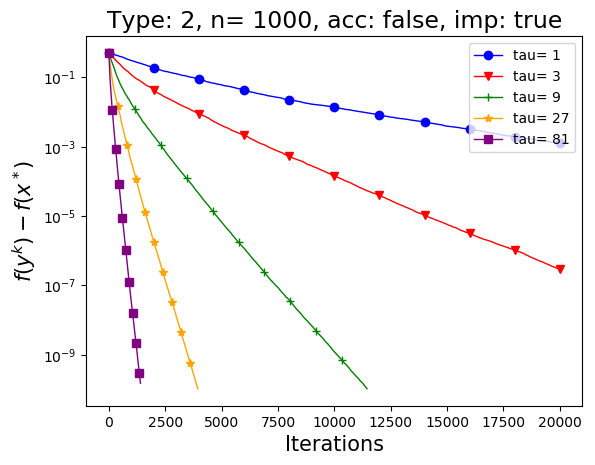}
\end{minipage}%
\begin{minipage}{0.25\textwidth}
  \centering
\includegraphics[width =  \textwidth ]{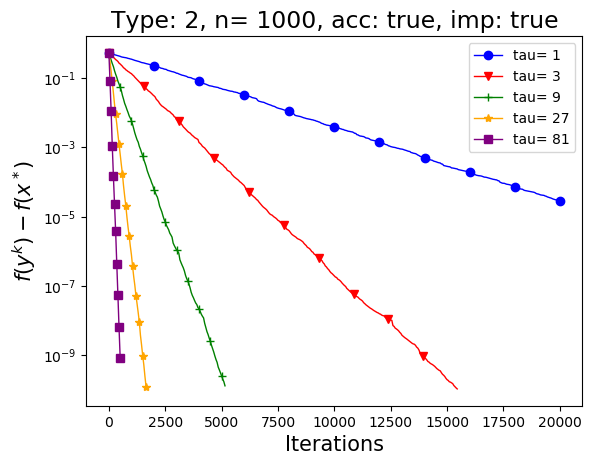}
\end{minipage}%
\\
\begin{minipage}{0.25\textwidth}
  \centering
\includegraphics[width =  \textwidth ]{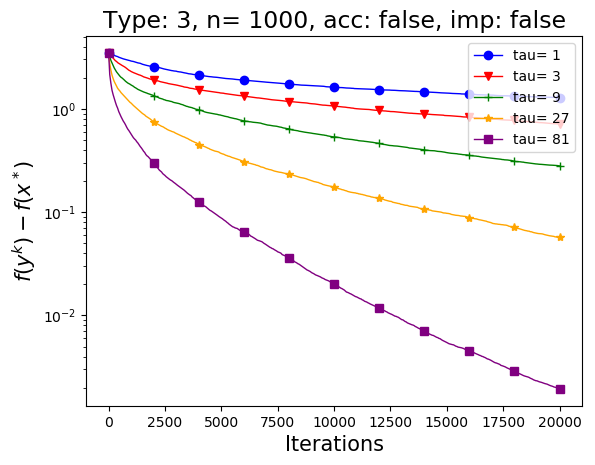}
\end{minipage}%
\begin{minipage}{0.25\textwidth}
  \centering
\includegraphics[width =  \textwidth ]{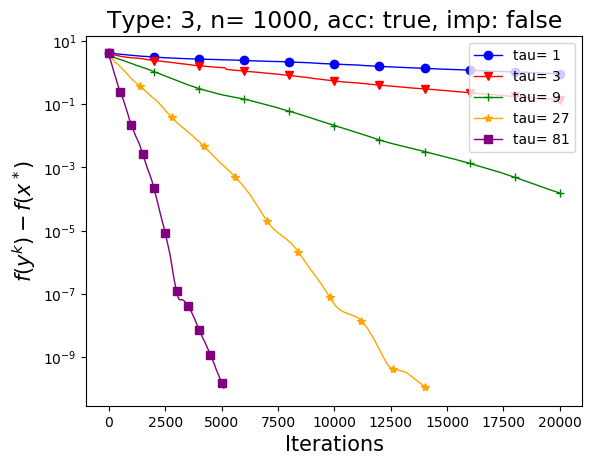}
\end{minipage}%
\begin{minipage}{0.25\textwidth}
  \centering
\includegraphics[width =  \textwidth ]{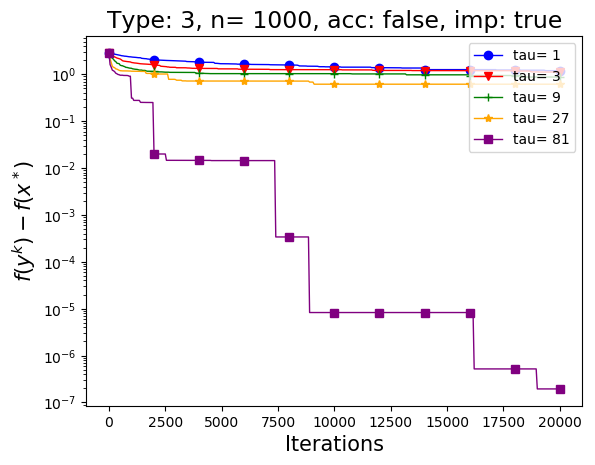}
\end{minipage}%
\begin{minipage}{0.25\textwidth}
  \centering
\includegraphics[width =  \textwidth ]{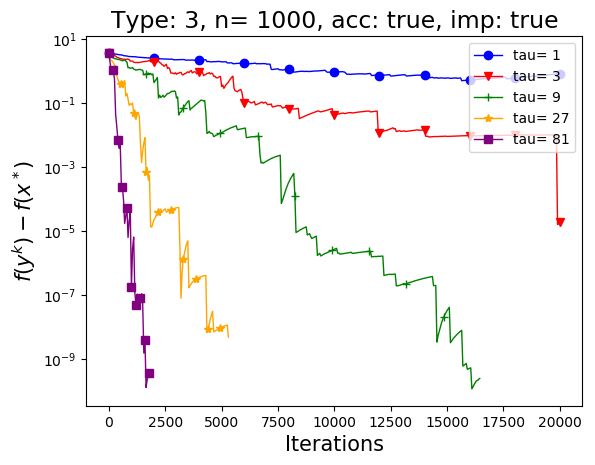}
\end{minipage}%
\\
\begin{minipage}{0.25\textwidth}
  \centering
\includegraphics[width =  \textwidth ]{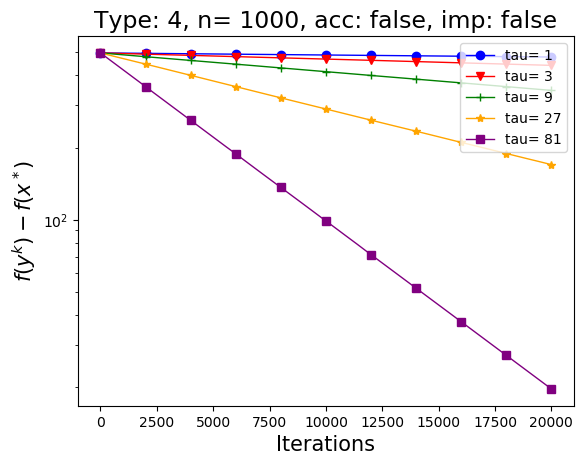}
\end{minipage}%
\begin{minipage}{0.25\textwidth}
  \centering
\includegraphics[width =  \textwidth ]{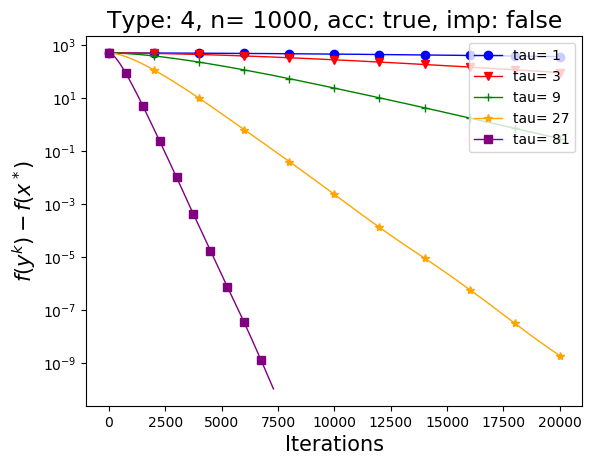}
\end{minipage}%
\begin{minipage}{0.25\textwidth}
  \centering
\includegraphics[width =  \textwidth ]{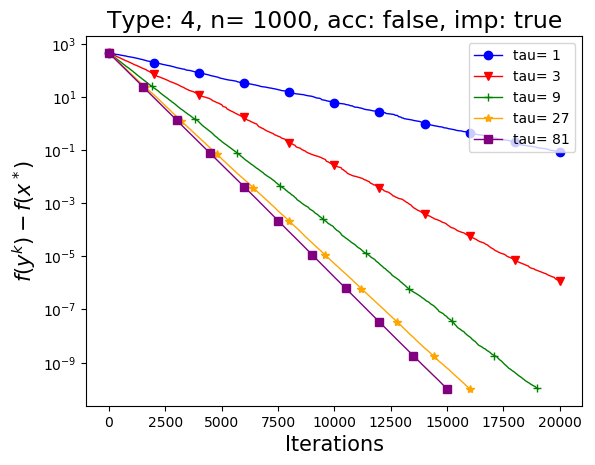}
\end{minipage}%
\begin{minipage}{0.25\textwidth}
  \centering
\includegraphics[width =  \textwidth ]{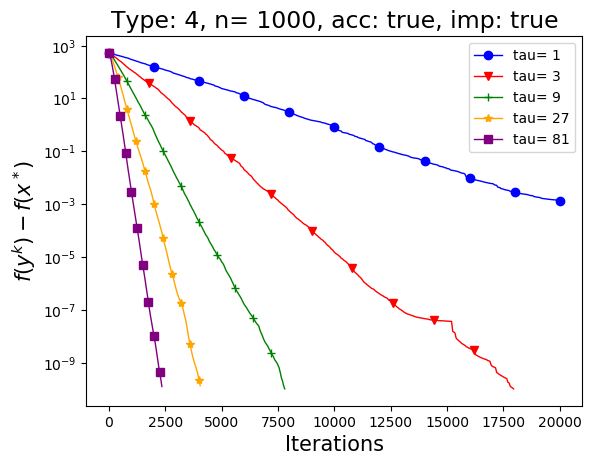}
\end{minipage}%
\\
\begin{minipage}{0.25\textwidth}
  \centering
\includegraphics[width =  \textwidth ]{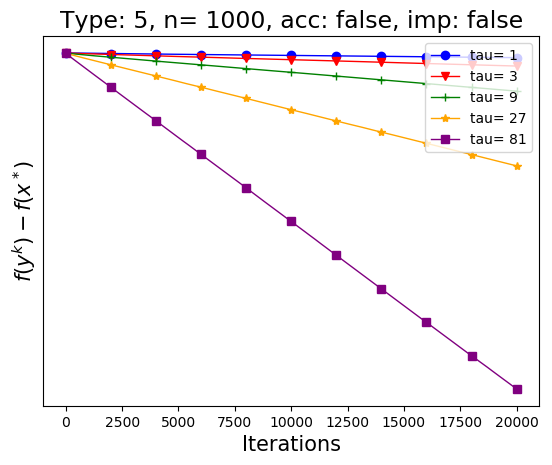}
\end{minipage}%
\begin{minipage}{0.25\textwidth}
  \centering
\includegraphics[width =  \textwidth ]{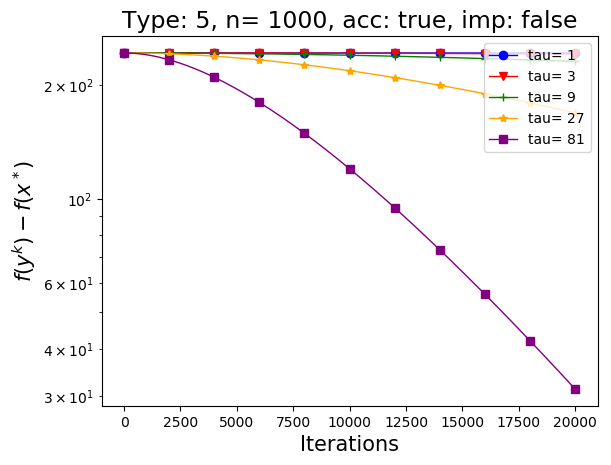}
\end{minipage}%
\begin{minipage}{0.25\textwidth}
  \centering
\includegraphics[width =  \textwidth ]{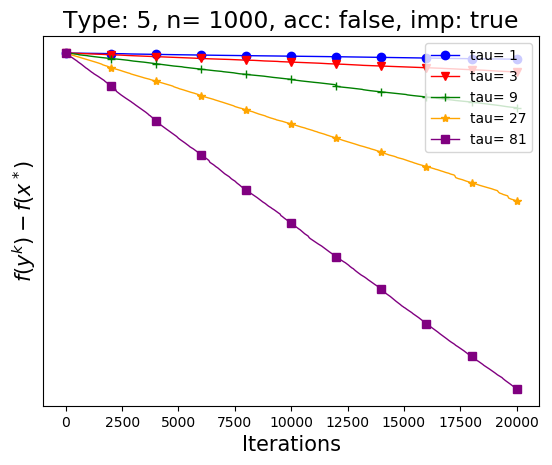}
\end{minipage}%
\begin{minipage}{0.25\textwidth}
  \centering
\includegraphics[width =  \textwidth ]{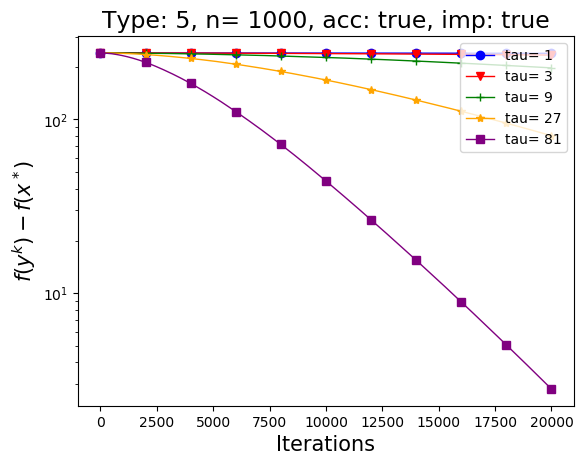}
\end{minipage}%
\\
\caption{Comparison of speedup gained by both $\tau$-nice sampling and importance sampling with and without acceleration on various quadratic problems.} \label{fig:tau_comp}
\end{figure}

\clearpage
\subsection{Logistic Regression \label{exp:logreg}}

In this section we apply \texttt{ACD} on the regularized logistic regression problem, i.e.
\[
f(x)= \frac1m \sum_{i=1}^m \log \left(1+\exp\left(\mA_{i,:}x\cdot  b\right) \right)+\frac{\lambda}{2} \| x\|^2,
\]
for $b\in \{-1,1 \}$ and data matrix $\mA$ comes from LibSVM. In each experiment in this section, we have chosen regularization parameter $\lambda$ to be the average diagonal element of the smoothness matrix. We first apply the methods with the optimal parameters as our theory suggests on smaller datasets. On larger ones (Section~\ref{sec:practical}), we set them in a cheaper way, which is not guaranteed to work by theory we provide. 

In our first experiment, we apply \texttt{ACD} on LibSVM data directly for various minibatch sizes $\tau$. Figure~\ref{fig:logreg_noncor} shows the results. As expected, \texttt{ACD} is always better to \texttt{CD}, and importance sampling is always better to uniform one. 
 
\begin{figure}[H]
\centering
\begin{minipage}{0.25\textwidth}
  \centering
\includegraphics[width =  \textwidth ]{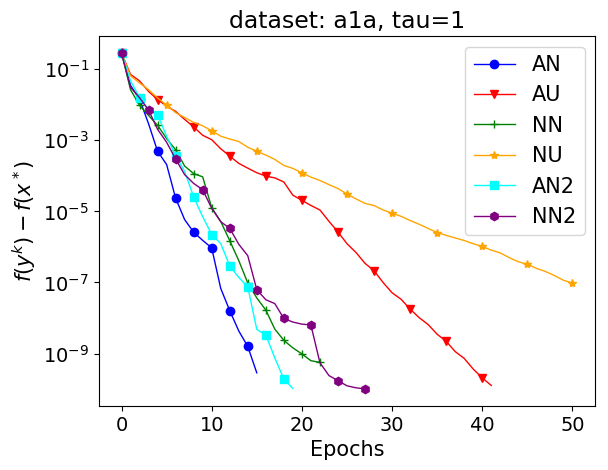}
\end{minipage}%
\begin{minipage}{0.25\textwidth}
  \centering
\includegraphics[width =  \textwidth ]{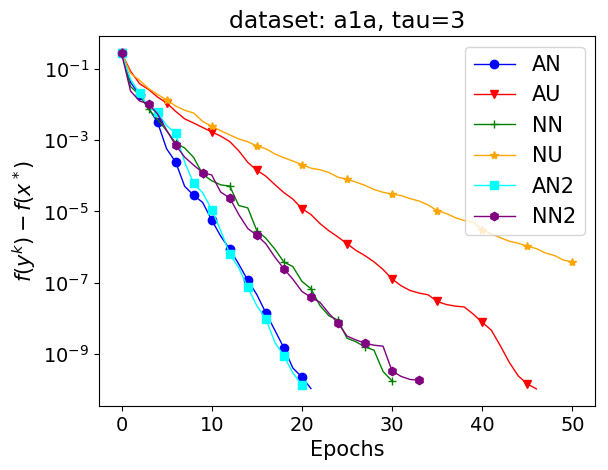}
\end{minipage}%
\begin{minipage}{0.25\textwidth}
  \centering
\includegraphics[width =  \textwidth ]{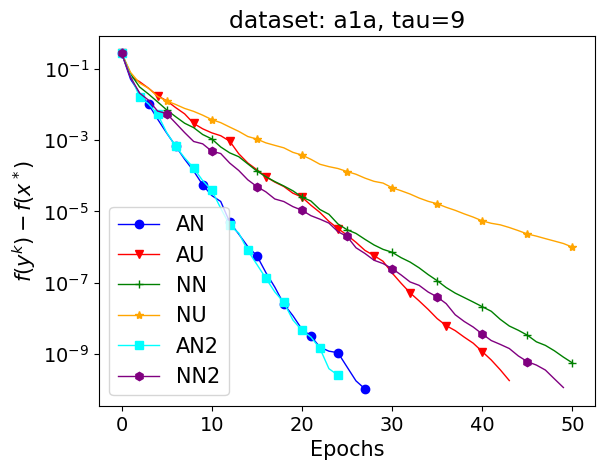}
\end{minipage}%
\begin{minipage}{0.25\textwidth}
  \centering
\includegraphics[width =  \textwidth ]{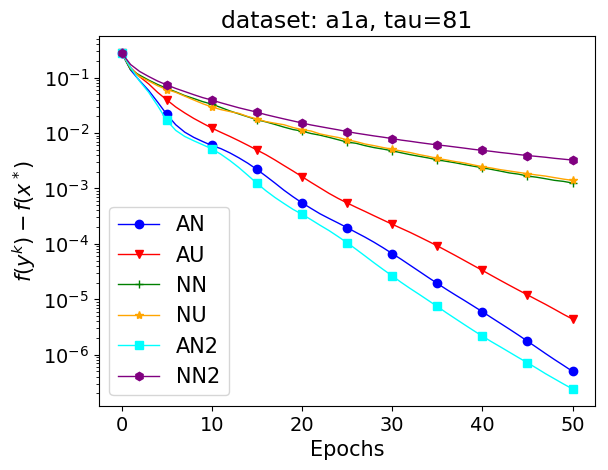}
\end{minipage}%
\\
\begin{minipage}{0.25\textwidth}
  \centering
\includegraphics[width =  \textwidth ]{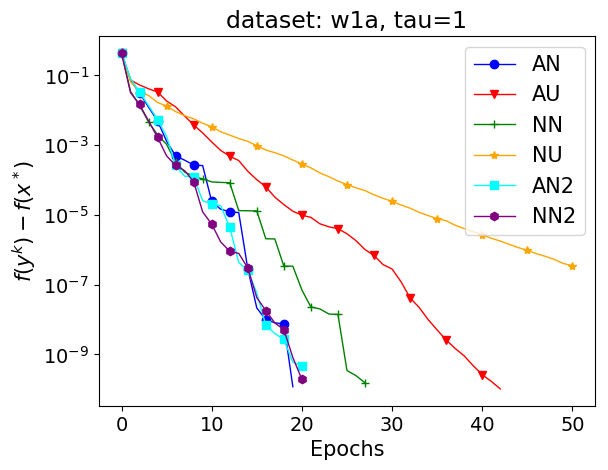}
\end{minipage}%
\begin{minipage}{0.25\textwidth}
  \centering
\includegraphics[width =  \textwidth ]{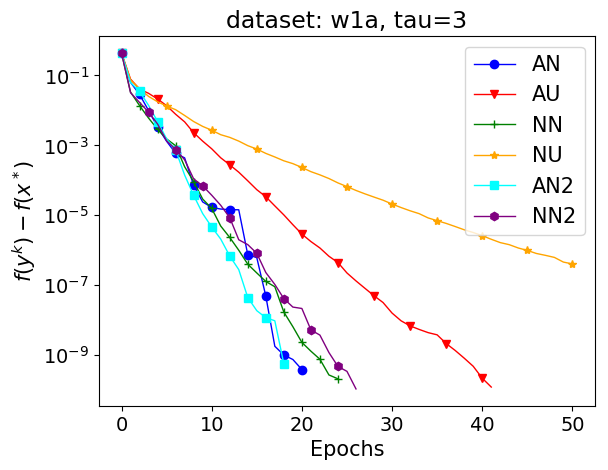}
\end{minipage}%
\begin{minipage}{0.25\textwidth}
  \centering
\includegraphics[width =  \textwidth ]{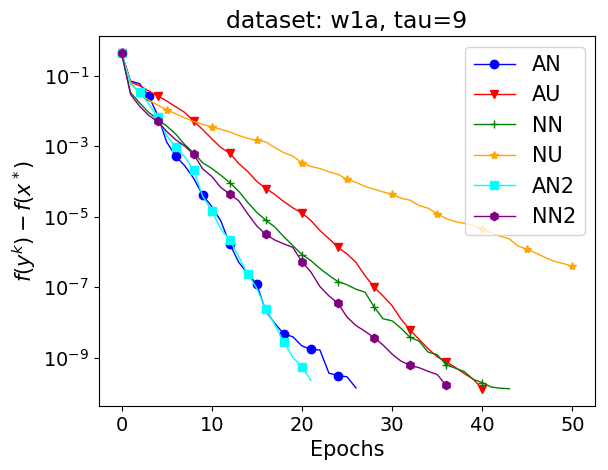}
\end{minipage}%
\begin{minipage}{0.25\textwidth}
  \centering
\includegraphics[width =  \textwidth ]{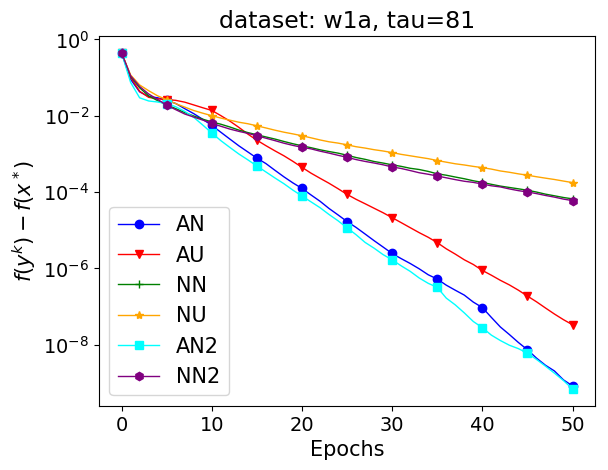}
\end{minipage}%
\\
\begin{minipage}{0.25\textwidth}
  \centering
\includegraphics[width =  \textwidth ]{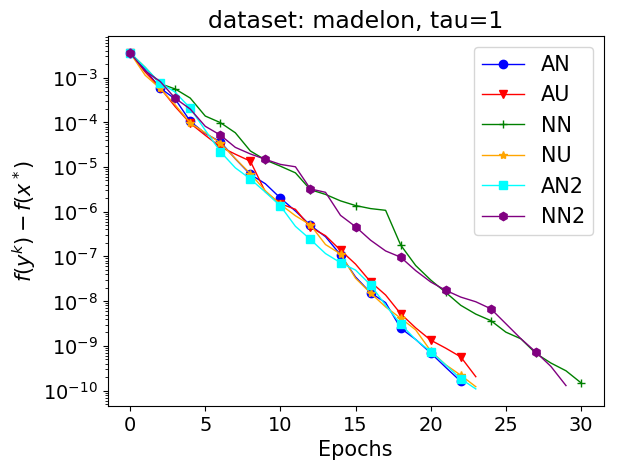}
\end{minipage}%
\begin{minipage}{0.25\textwidth}
  \centering
\includegraphics[width =  \textwidth ]{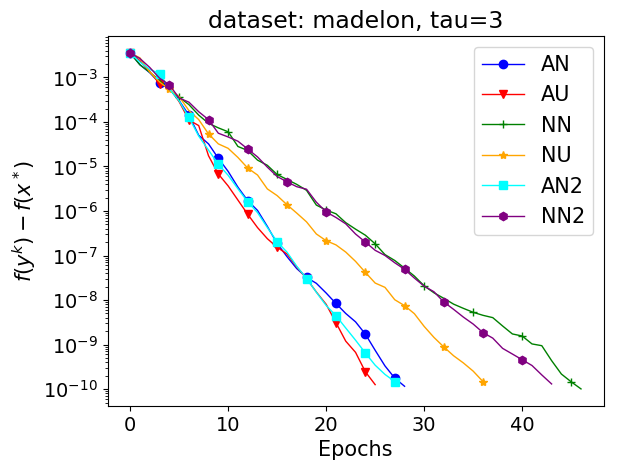}
\end{minipage}%
\begin{minipage}{0.25\textwidth}
  \centering
\includegraphics[width =  \textwidth ]{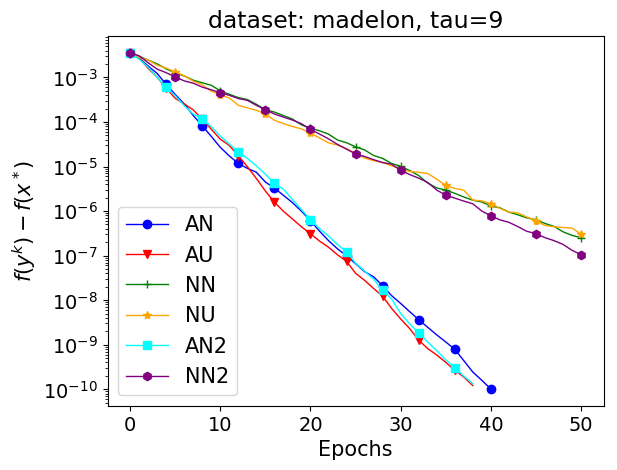}
\end{minipage}%
\begin{minipage}{0.25\textwidth}
  \centering
\includegraphics[width =  \textwidth ]{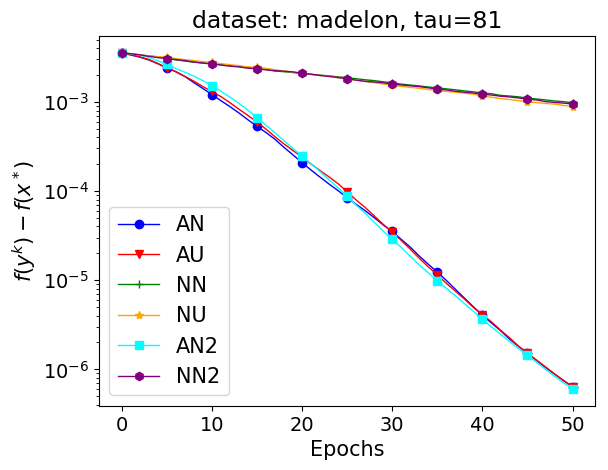}
\end{minipage}%
\\
\centering
\begin{minipage}{0.25\textwidth}
  \centering
\includegraphics[width =  \textwidth ]{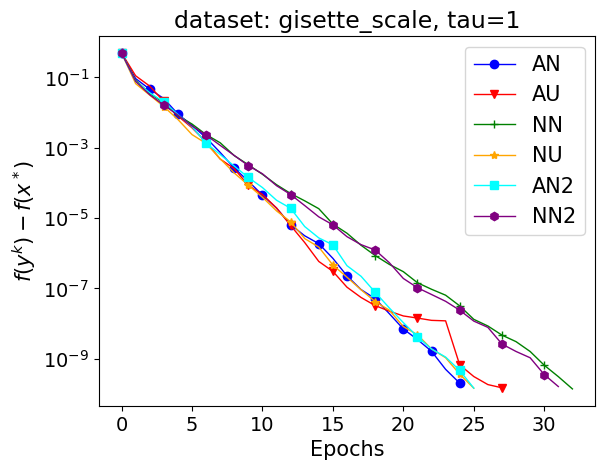}
\end{minipage}%
\begin{minipage}{0.25\textwidth}
  \centering
\includegraphics[width =  \textwidth ]{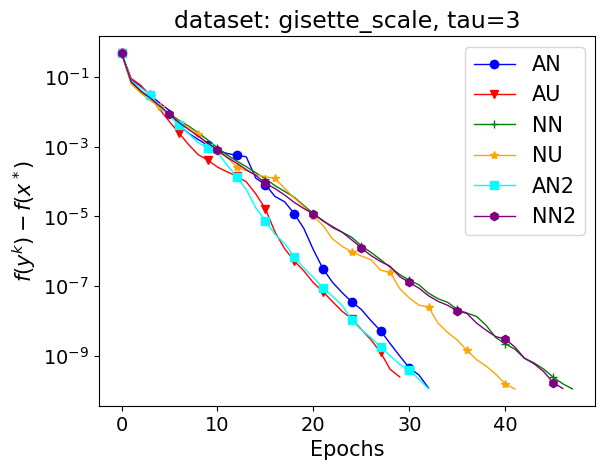}
\end{minipage}%
\begin{minipage}{0.25\textwidth}
  \centering
\includegraphics[width =  \textwidth ]{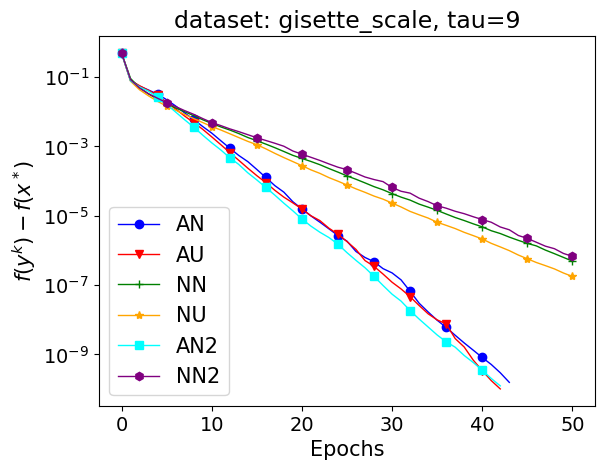}
\end{minipage}%
\begin{minipage}{0.25\textwidth}
  \centering
\includegraphics[width =  \textwidth ]{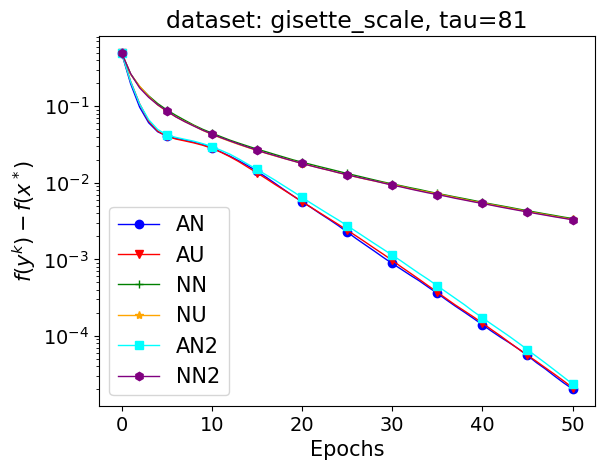}
\end{minipage}%
\\
\begin{minipage}{0.25\textwidth}
  \centering
\includegraphics[width =  \textwidth ]{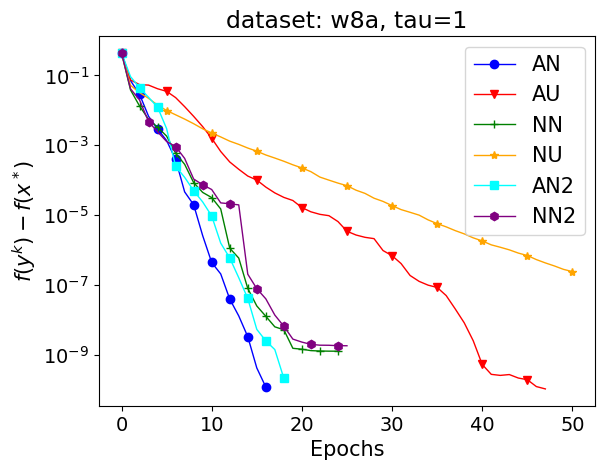}
\end{minipage}%
\begin{minipage}{0.25\textwidth}
  \centering
\includegraphics[width =  \textwidth ]{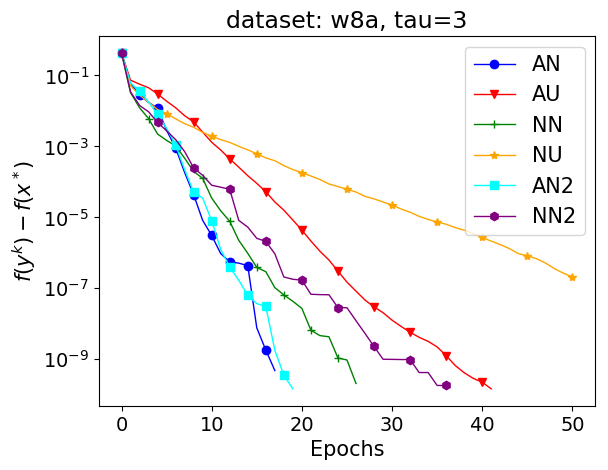}
\end{minipage}%
\begin{minipage}{0.25\textwidth}
  \centering
\includegraphics[width =  \textwidth ]{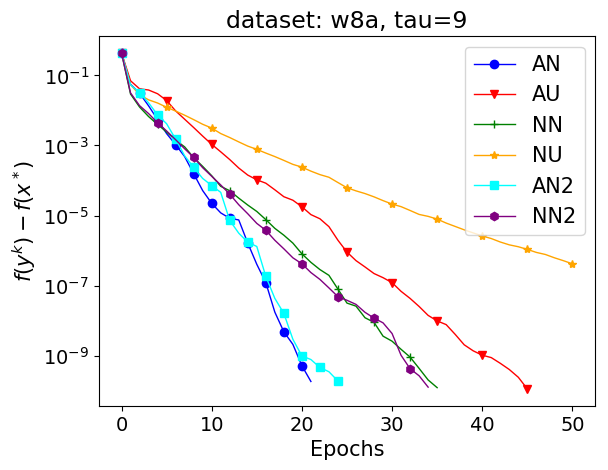}
\end{minipage}%
\begin{minipage}{0.25\textwidth}
  \centering
\includegraphics[width =  \textwidth ]{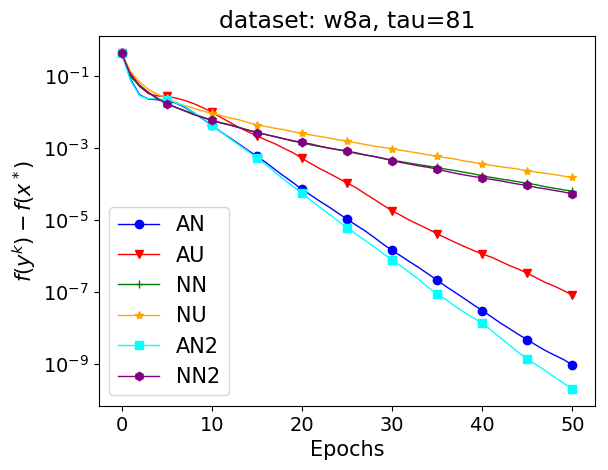}
\end{minipage}%
\caption{Accelerated coordinate desent applied on the logistic regression problem, for various LibSVM datasets and minibatch sizes $\tau$}\label{fig:logreg_noncor}
\end{figure}

Note that, for for some datasets and especially bigger minibatch sizes, the effect of importance sampling is sometimes negligible. To demonstrate the power of importance sampling, in the next experiment, we first corrupt the data -- we multiply each row and column of the data matrix $\mA$ by random number from uniform distribution over $[0,1]$. The results can be seen in Figure~\ref{fig:logreg_cor}. As expected, the effect of importance sampling becomes more significant.

\begin{figure}[H]
\centering
\begin{minipage}{0.25\textwidth}
  \centering
\includegraphics[width =  \textwidth ]{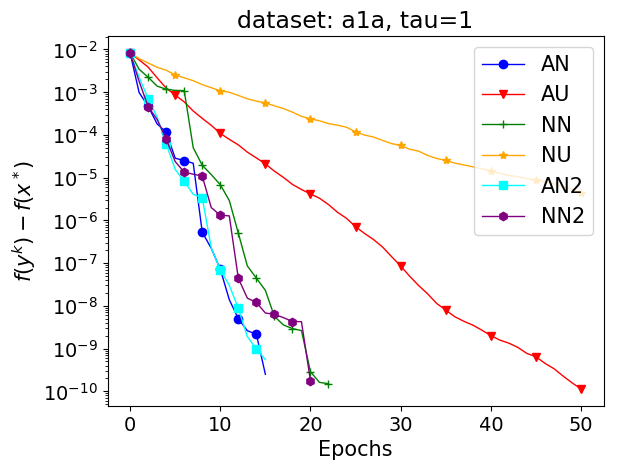}
\end{minipage}%
\begin{minipage}{0.25\textwidth}
  \centering
\includegraphics[width =  \textwidth ]{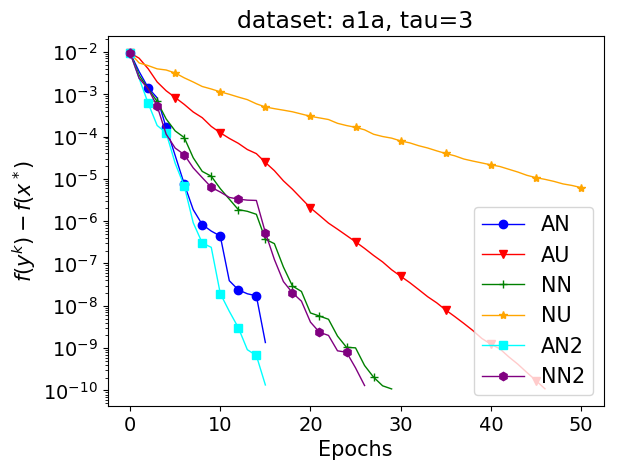}
\end{minipage}%
\begin{minipage}{0.25\textwidth}
  \centering
\includegraphics[width =  \textwidth ]{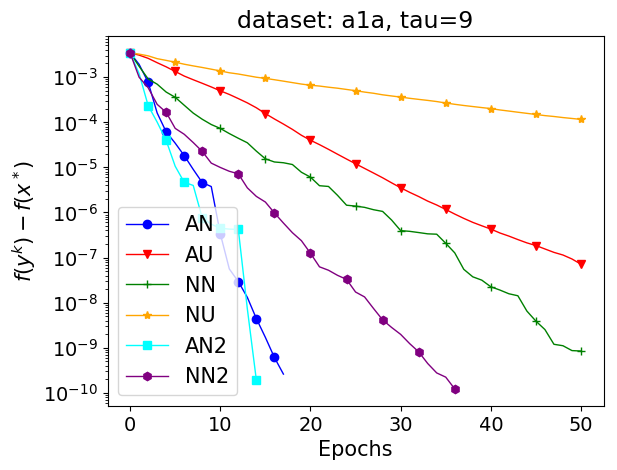}
\end{minipage}%
\begin{minipage}{0.25\textwidth}
  \centering
\includegraphics[width =  \textwidth ]{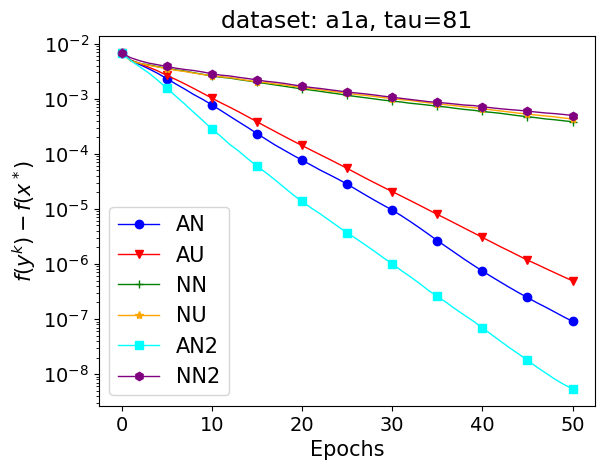}
\end{minipage}%
\\
\begin{minipage}{0.25\textwidth}
  \centering
\includegraphics[width =  \textwidth ]{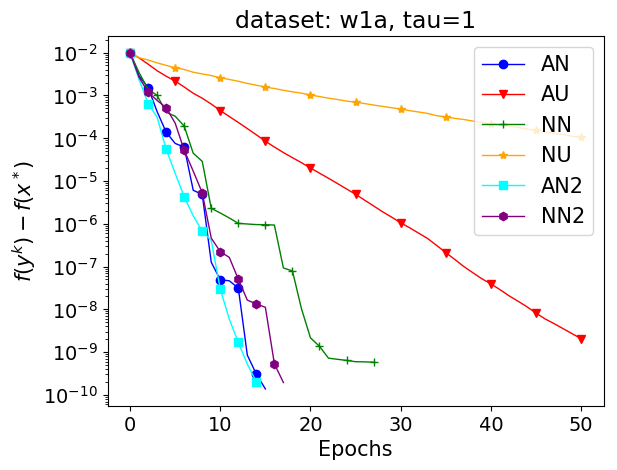}
\end{minipage}%
\begin{minipage}{0.25\textwidth}
  \centering
\includegraphics[width =  \textwidth ]{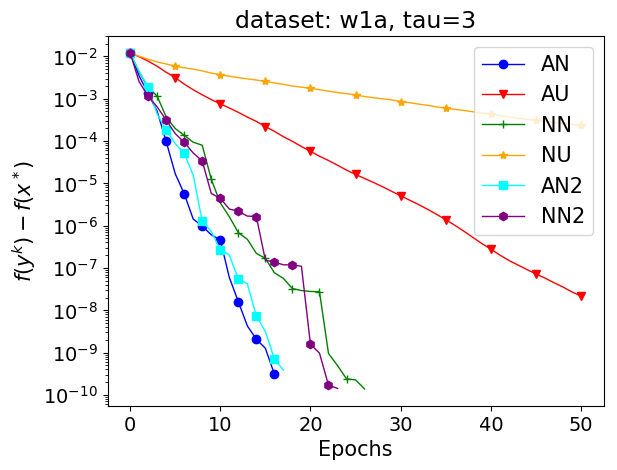}
\end{minipage}%
\begin{minipage}{0.25\textwidth}
  \centering
\includegraphics[width =  \textwidth ]{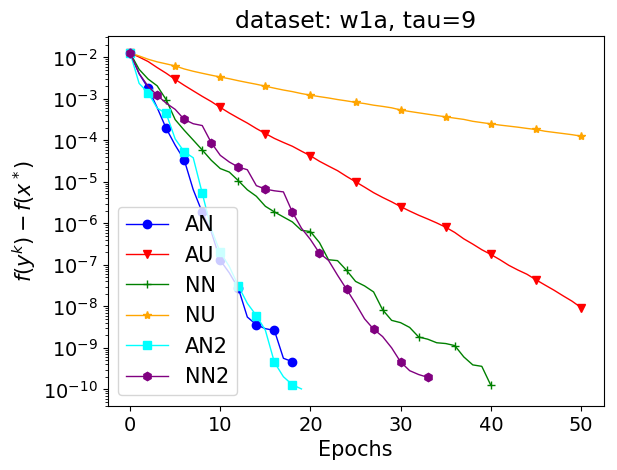}
\end{minipage}%
\begin{minipage}{0.25\textwidth}
  \centering
\includegraphics[width =  \textwidth ]{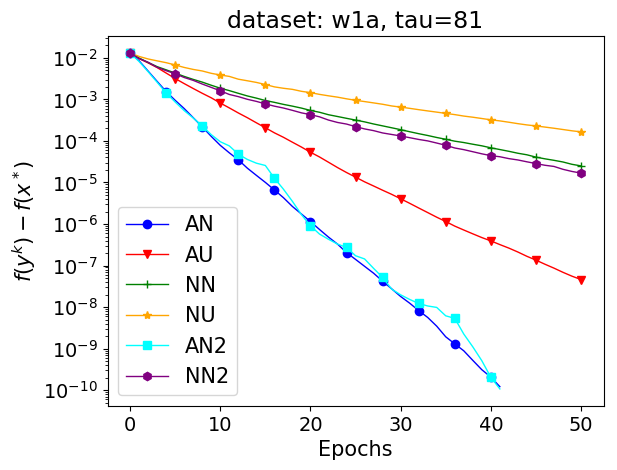}
\end{minipage}%
\\
\begin{minipage}{0.25\textwidth}
  \centering
\includegraphics[width =  \textwidth ]{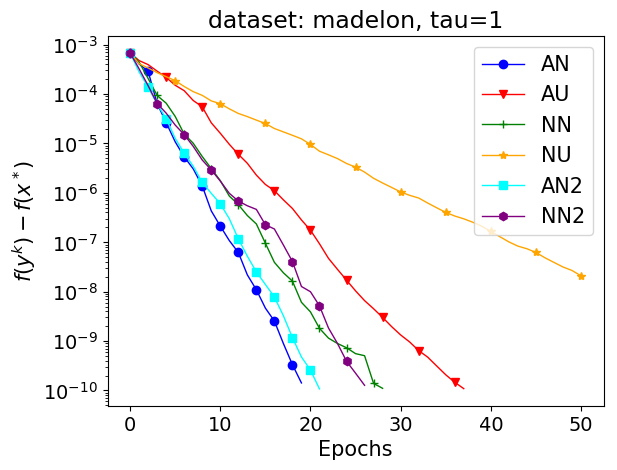}
\end{minipage}%
\begin{minipage}{0.25\textwidth}
  \centering
\includegraphics[width =  \textwidth ]{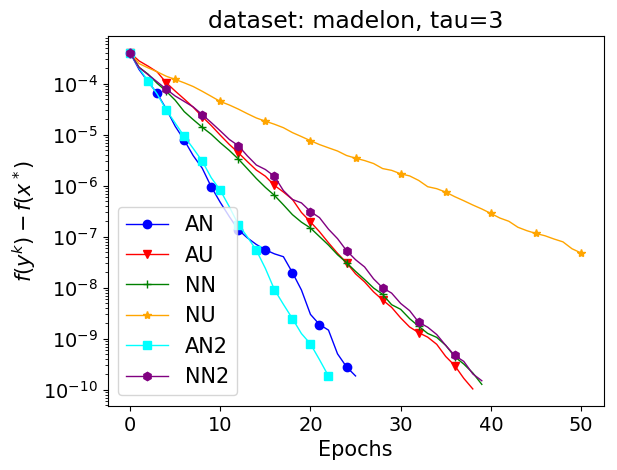}
\end{minipage}%
\begin{minipage}{0.25\textwidth}
  \centering
\includegraphics[width =  \textwidth ]{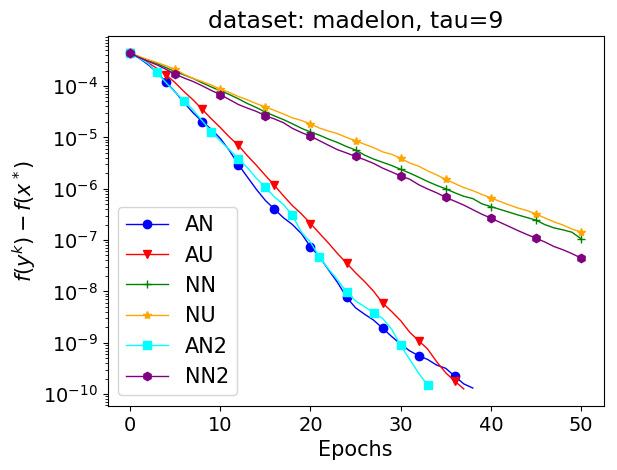}
\end{minipage}%
\begin{minipage}{0.25\textwidth}
  \centering
\includegraphics[width =  \textwidth ]{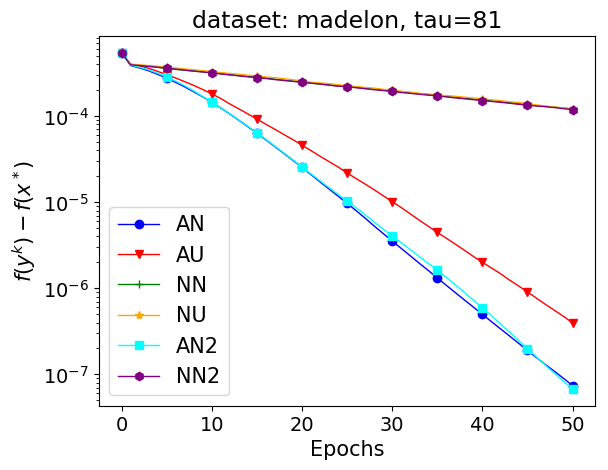}
\end{minipage}%
\\
\centering
\begin{minipage}{0.25\textwidth}
  \centering
\includegraphics[width =  \textwidth ]{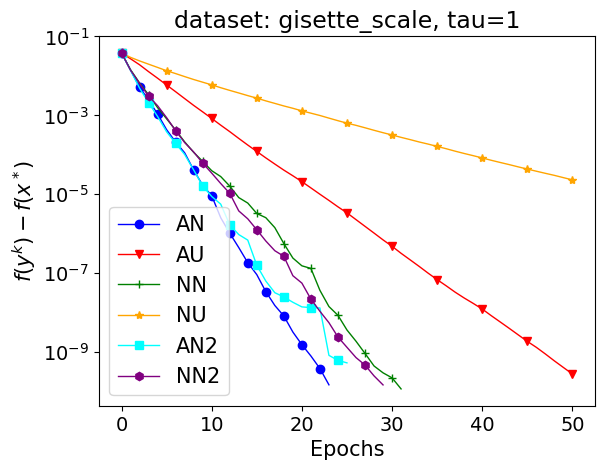}
\end{minipage}%
\begin{minipage}{0.25\textwidth}
  \centering
\includegraphics[width =  \textwidth ]{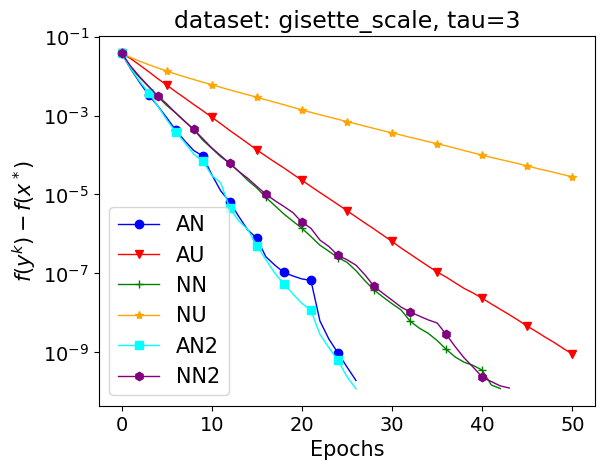}
\end{minipage}%
\begin{minipage}{0.25\textwidth}
  \centering
\includegraphics[width =  \textwidth ]{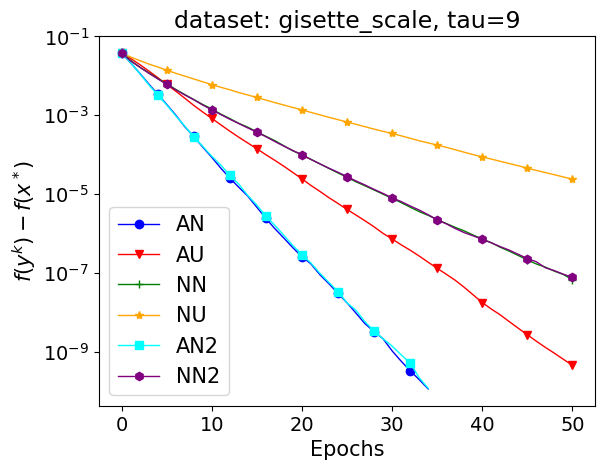}
\end{minipage}%
\begin{minipage}{0.25\textwidth}
  \centering
\includegraphics[width =  \textwidth ]{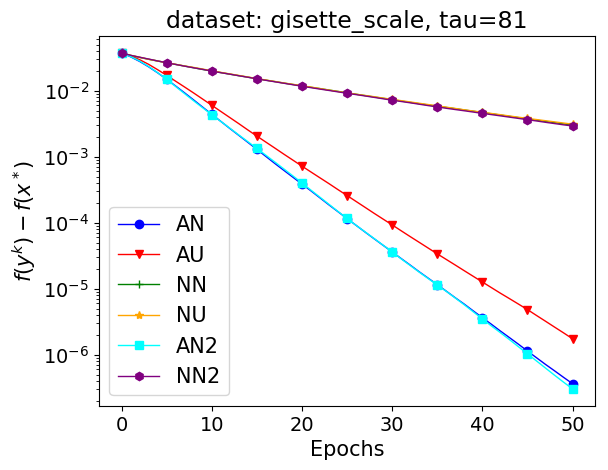}
\end{minipage}%
\\
\begin{minipage}{0.25\textwidth}
  \centering
\includegraphics[width =  \textwidth ]{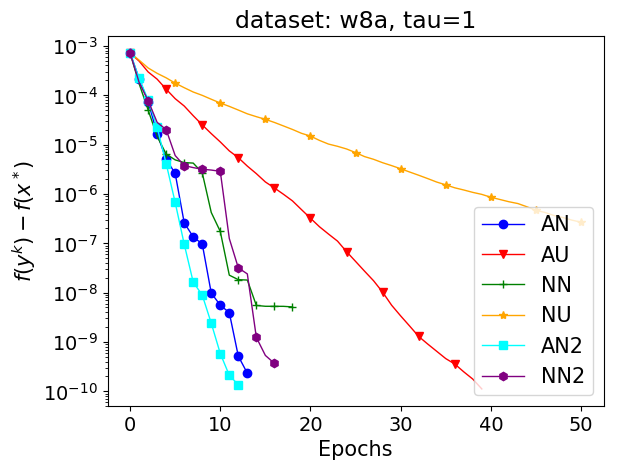}
\end{minipage}%
\begin{minipage}{0.25\textwidth}
  \centering
\includegraphics[width =  \textwidth ]{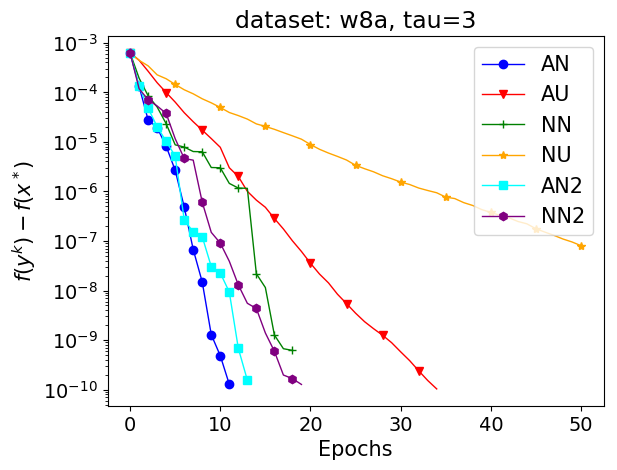}
\end{minipage}%
\begin{minipage}{0.25\textwidth}
  \centering
\includegraphics[width =  \textwidth ]{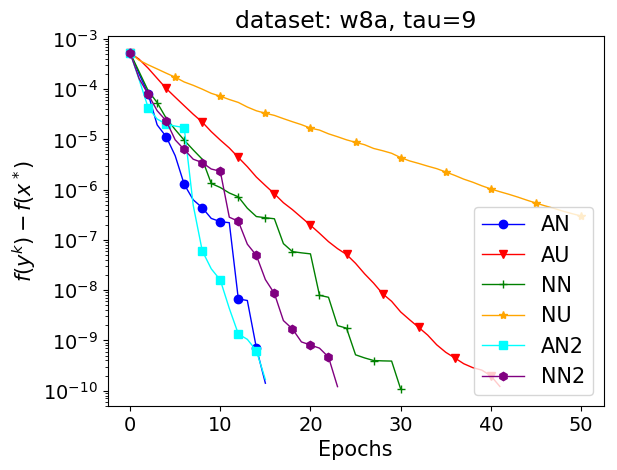}
\end{minipage}%
\begin{minipage}{0.25\textwidth}
  \centering
\includegraphics[width =  \textwidth ]{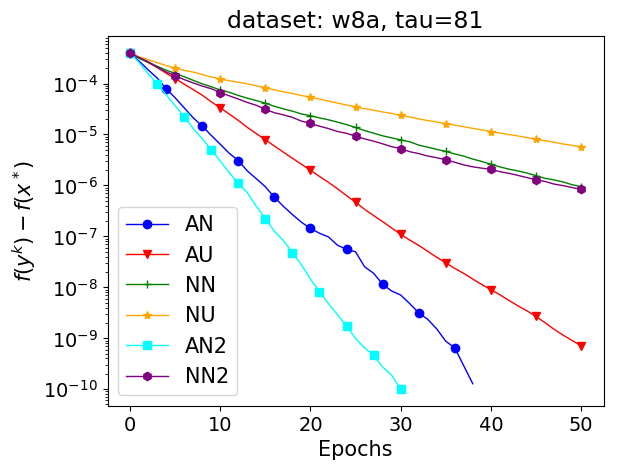}
\end{minipage}%
\caption{Accelerated coordinate desent applied on the logistic regression problem, for various rescaled LibSVM datasets and minibatch sizes $\tau$}\label{fig:logreg_cor}
\end{figure}

\clearpage
\subsubsection{Practical method on larger dataset \label{sec:practical}}
For completeness, we restate here experiments from Figure~\ref{fig:logreg_big}. We have chosen regularization parameter $\lambda$ to be the average diagonal element of the smoothness matrix and estimated $v,\sigma$ as described in Section~\ref{sec:exp}. 

\begin{figure}[H]
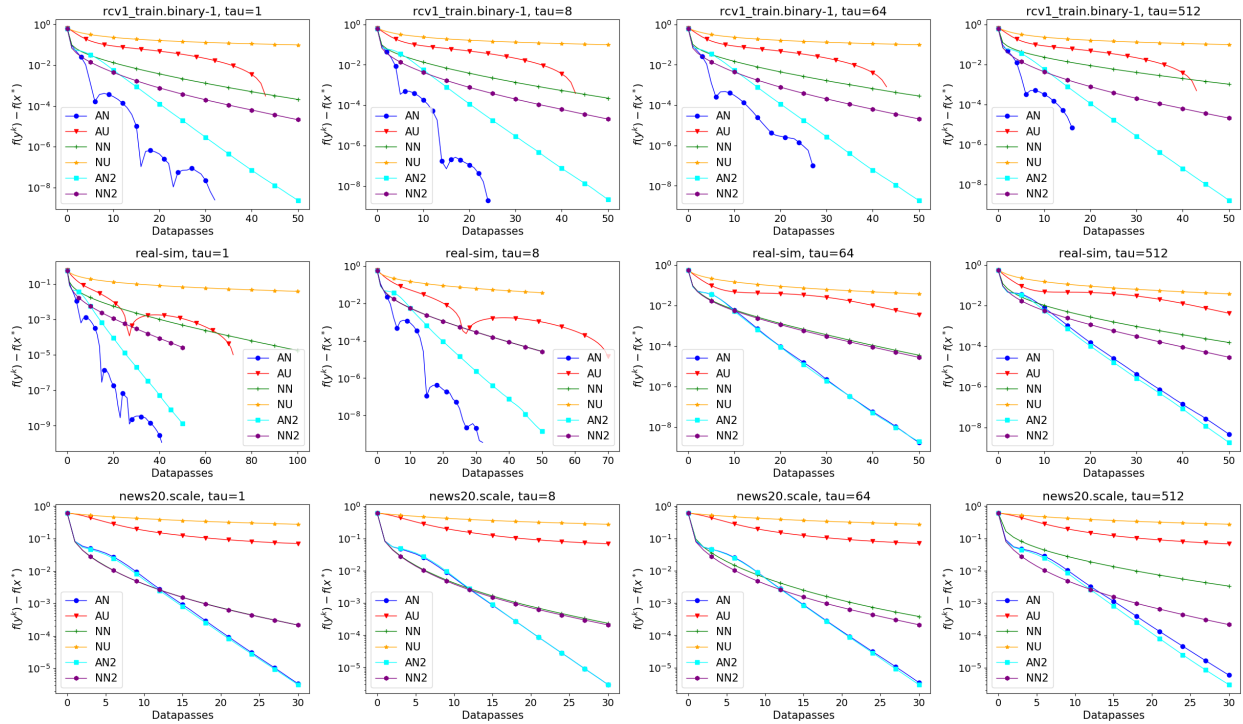

\centering
\begin{minipage}{0.25\textwidth}
  \centering
\includegraphics[width =  \textwidth ]{NEW99_REMOTE_tau_1_data_rcv.png}
\end{minipage}%
\begin{minipage}{0.25\textwidth}
  \centering
\includegraphics[width =  \textwidth ]{NEW99_REMOTE_tau_8_data_rcv.png}
\end{minipage}%
\begin{minipage}{0.25\textwidth}
  \centering
\includegraphics[width =  \textwidth ]{NEW99_REMOTE_tau_64_data_rcv.png}
\end{minipage}%
\begin{minipage}{0.25\textwidth}
  \centering
\includegraphics[width =  \textwidth ]{NEW99_REMOTE_tau_512_data_rcv.png}
\end{minipage}%
\\
\begin{minipage}{0.25\textwidth}
  \centering
\includegraphics[width =  \textwidth ]{NEW99_REMOTE_tau_1_data_rea.png}
\end{minipage}%
\begin{minipage}{0.25\textwidth}
  \centering
\includegraphics[width =  \textwidth ]{NEW99_REMOTE_tau_8_data_rea.png}
\end{minipage}%
\begin{minipage}{0.25\textwidth}
  \centering
\includegraphics[width =  \textwidth ]{NEW99_REMOTE_tau_64_data_rea.png}
\end{minipage}%
\begin{minipage}{0.25\textwidth}
  \centering
\includegraphics[width =  \textwidth ]{NEW99_REMOTE_tau_512_data_rea.png}
\end{minipage}%
\\
\begin{minipage}{0.25\textwidth}
  \centering
\includegraphics[width =  \textwidth ]{NEW99_REMOTE_tau_1_data_new.png}
\end{minipage}%
\begin{minipage}{0.25\textwidth}
  \centering
\includegraphics[width =  \textwidth ]{NEW99_REMOTE_tau_8_data_new.png}
\end{minipage}%
\begin{minipage}{0.25\textwidth}
  \centering
\includegraphics[width =  \textwidth ]{NEW99_REMOTE_tau_64_data_new.png}
\end{minipage}%
\begin{minipage}{0.25\textwidth}
  \centering
\includegraphics[width =  \textwidth ]{NEW99_REMOTE_tau_512_data_new.png}
\end{minipage}%
\caption{Six variants of coordinate descent (\texttt{AN}, \texttt{AU}, \texttt{NN}, \texttt{NU}, \texttt{AN2} and \texttt{AU2})  applied to a logistic regression problem, with minibatch sizes $\tau=1, 8, 64$ and $512$.}\label{fig:logreg_big2}
\end{figure}

\clearpage
\subsection{Support Vector Machines \label{sec:SVM}}
In this section we apply \texttt{ACD} on the dual of  SVM problem with squared hinge loss, i.e.,
\[
f(x)= \frac{1}{\lambda n^2} \sum_{j=1}^m\left( \sum_{i=1}^n b_i \mA_{ji} x_i\right)^2-\frac1n \sum_{i=1}^n x_i +\frac{1}{4n}\sum_{i=1}^n x_i^2 + \cI_{[0,\infty]}(x),
\]
where $\cI_{[0,\infty]}$ stands for indicator function of set $[0,\infty]$, i.e. $\cI_{[0,\infty]}(x)=0$ if $x\in \R^n_+$, otherwise $\cI_{[0,\infty]}(x)=\infty$. As for the data, we have rescaled each row and each column of the data matrix coming frol LibSVM by random scalar generated from uniform distribution over $[0,1]$. We have chosen regularization parameter $\lambda$ to be maximal diagonal element of the smoothness matrix divided by 10 in each experiment below. We deal with nonsmooth indicator function using proximal operator, which happens to be a projection in this case. We choose ESO parameters $v$ from Lemma~\ref{thm:special-ESO-result}, while estimating the smoothness matrix as $\sqrt{n}$--times multiple of its diagonal. An estimate of the strong convexity $\sigma$ for acceleration was chosen to be minimal diagonal element of the smoothness matrix, therefore we adapt a similar approach as in Section~\ref{sec:practical}.

Recall that we did not provide a theory for the proximal steps. However, we make the experiment to demonstrate that \texttt{ACD} can solve big data problems on top of large dimensional problems. The results (Figure~\ref{fig:SVM}) again suggests the great significance of the acceleration and importance sampling.

\begin{figure}[H]
\centering
\begin{minipage}{0.25\textwidth}
  \centering
\includegraphics[width =  \textwidth ]{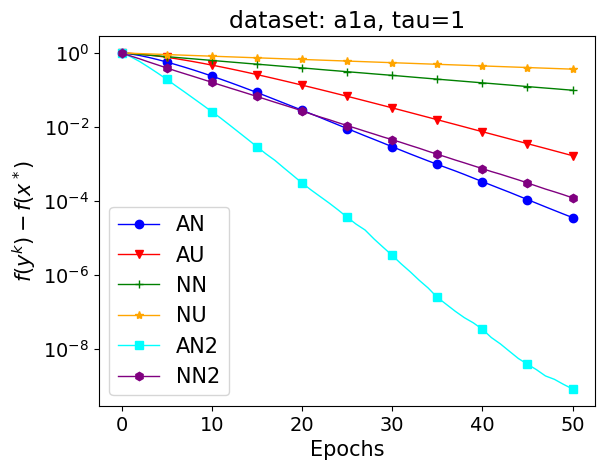}
\end{minipage}%
\begin{minipage}{0.25\textwidth}
  \centering
\includegraphics[width =  \textwidth ]{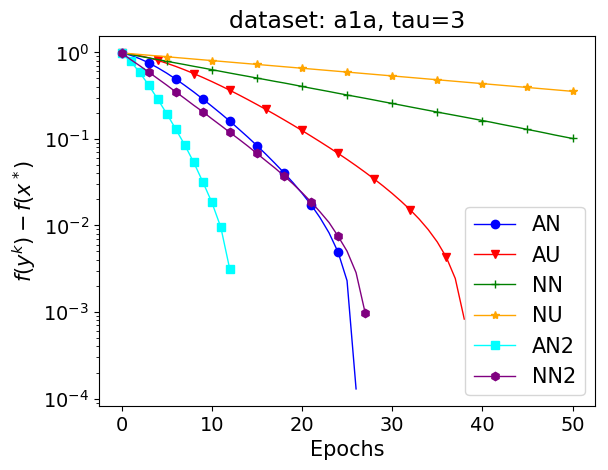}
\end{minipage}%
\begin{minipage}{0.25\textwidth}
  \centering
\includegraphics[width =  \textwidth ]{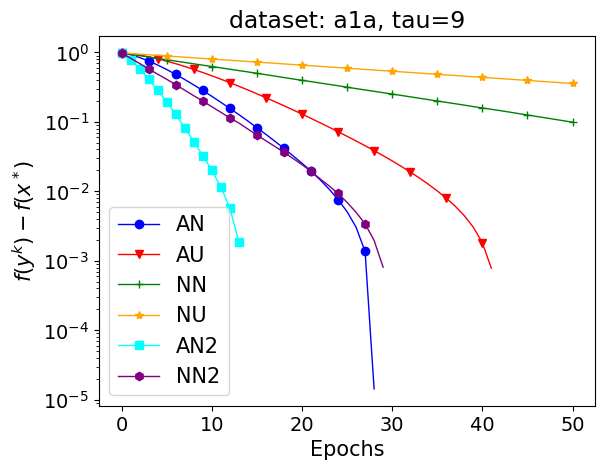}
\end{minipage}%
\begin{minipage}{0.25\textwidth}
  \centering
\includegraphics[width =  \textwidth ]{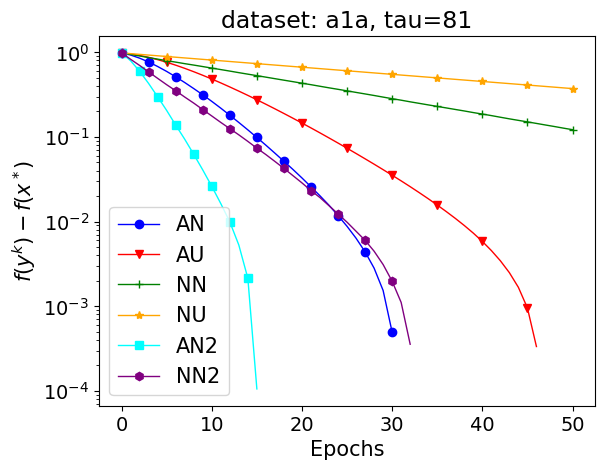}
\end{minipage}%
\\
\begin{minipage}{0.25\textwidth}
  \centering
\includegraphics[width =  \textwidth ]{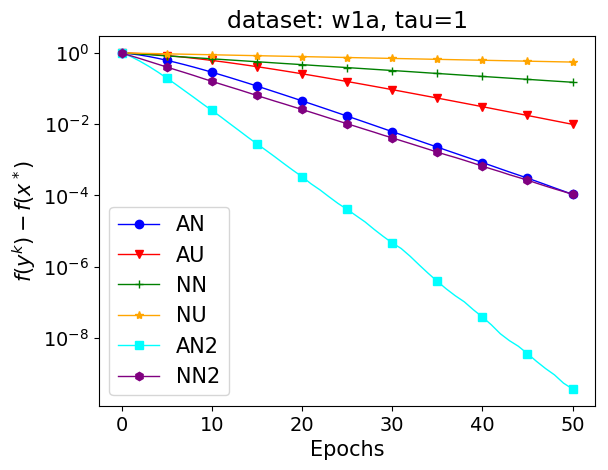}
\end{minipage}%
\begin{minipage}{0.25\textwidth}
  \centering
\includegraphics[width =  \textwidth ]{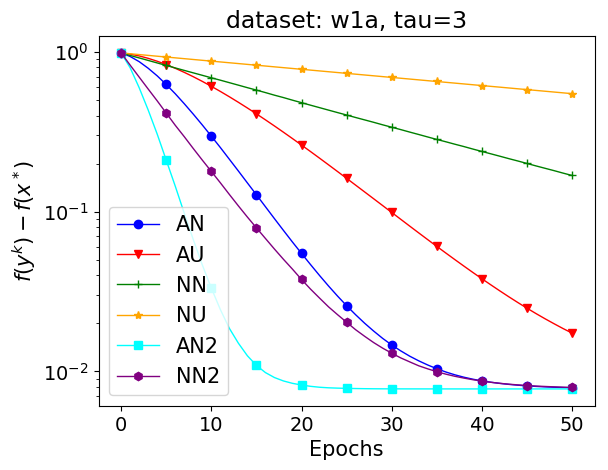}
\end{minipage}%
\begin{minipage}{0.25\textwidth}
  \centering
\includegraphics[width =  \textwidth ]{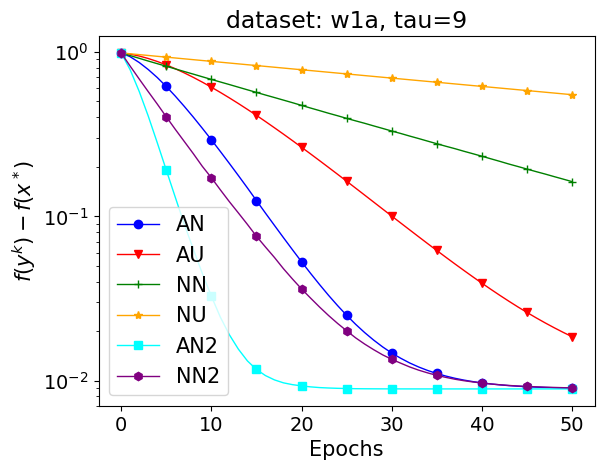}
\end{minipage}%
\begin{minipage}{0.25\textwidth}
  \centering
\includegraphics[width =  \textwidth ]{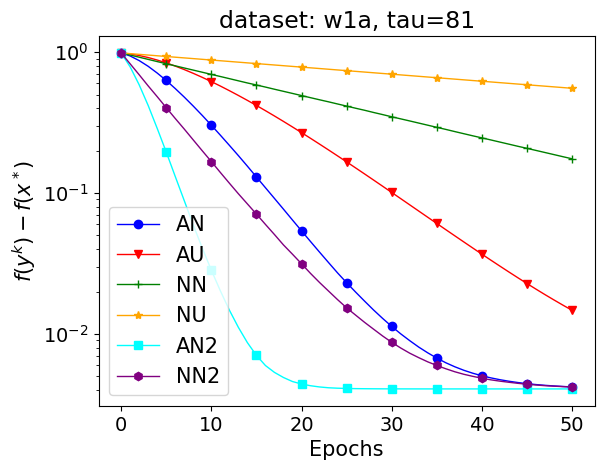}
\end{minipage}%
\\
\begin{minipage}{0.25\textwidth}
  \centering
\includegraphics[width =  \textwidth ]{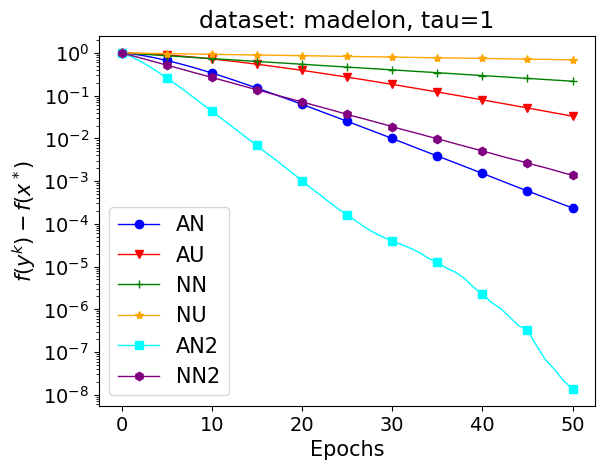}
\end{minipage}%
\begin{minipage}{0.25\textwidth}
  \centering
\includegraphics[width =  \textwidth ]{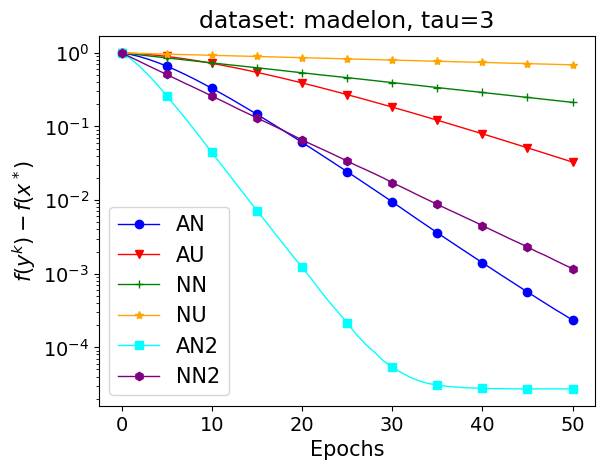}
\end{minipage}%
\begin{minipage}{0.25\textwidth}
  \centering
\includegraphics[width =  \textwidth ]{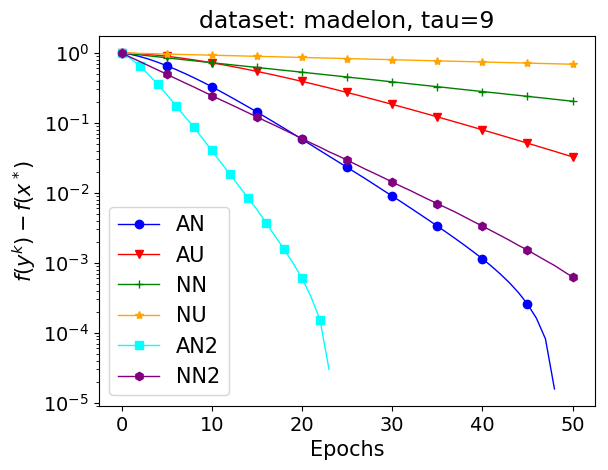}
\end{minipage}%
\begin{minipage}{0.25\textwidth}
  \centering
\includegraphics[width =  \textwidth ]{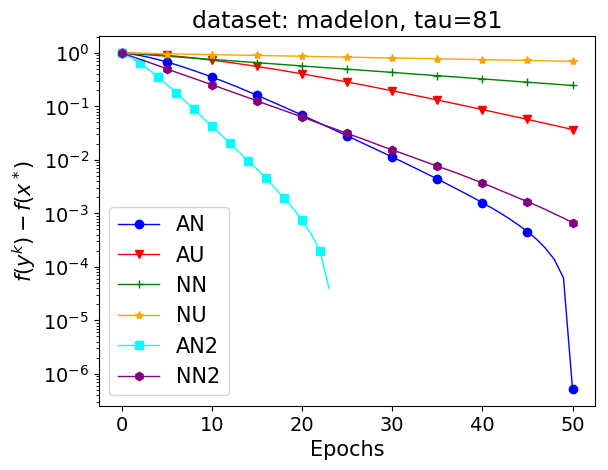}
\end{minipage}%
\\
\centering
\begin{minipage}{0.25\textwidth}
  \centering
\includegraphics[width =  \textwidth ]{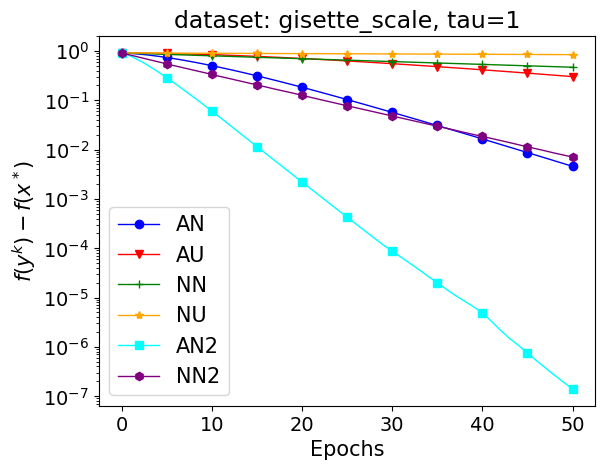}
\end{minipage}%
\begin{minipage}{0.25\textwidth}
  \centering
\includegraphics[width =  \textwidth ]{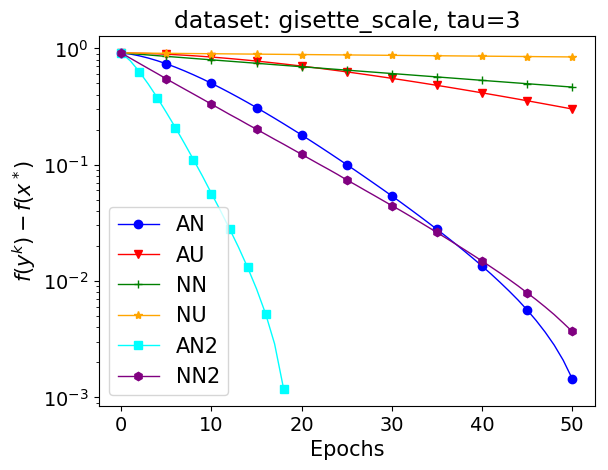}
\end{minipage}%
\begin{minipage}{0.25\textwidth}
  \centering
\includegraphics[width =  \textwidth ]{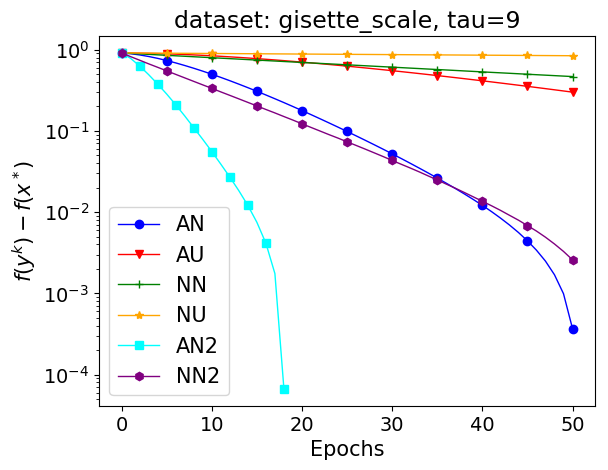}
\end{minipage}%
\begin{minipage}{0.25\textwidth}
  \centering
\includegraphics[width =  \textwidth ]{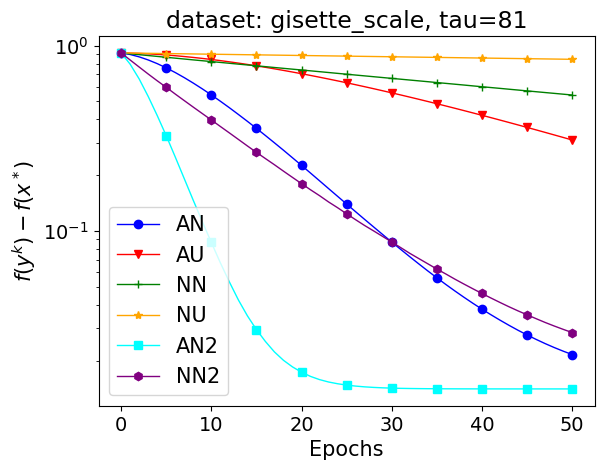}
\end{minipage}%
\caption{Accelerated coordinate desent applied on the dual of of SVM with squared hinge loss, for various LibSVM datasets } \label{fig:SVM}
\end{figure}

\clearpage

\end{document}